\documentclass[11pt]{amsart}
\usepackage{amsfonts,amssymb,amscd}
\usepackage[mathscr]{eucal}

\newcommand{\e}{\varepsilon}

\newcommand{\dist}{\mathrm{dist}\,}
\newcommand{\diam}{\mathrm{diam}\,}

\newcommand{\IN}{\mathbb N}
\newcommand{\IZ}{\mathbb Z}

\newcommand{\IR}{\mathbb R}

\newcommand{\EC}{\mathsf{E\!C}}
\newcommand{\Ra}{\Rightarrow}
\newcommand{\w}{\omega}
\newcommand{\lev}{\mathrm{lev}}
\newcommand{\Deg}{\mathrm{Deg}}
\newcommand{\mesh}{\mathrm{mesh}\,}
\newcommand{\C}{\mathscr C}
\newcommand{\A}{\mathcal A}

\newcommand{\suc}{\mathrm{pred}}

\newcommand{\Ent}{\Theta}
\newcommand{\ent}{\theta}
\newcommand{\id}{\mathrm{i\kern-1pt d}}

\newcommand{\upa}{\uparrow}
\newcommand{\da}{\downarrow}
\newcommand{\dupa}{\updownarrow}
\newcommand{\Lev}{\mathrm{Lev}}
\newcommand{\pr}{\mathrm{pr}}

\newcommand{\vx}{\mathrm{vx}}

\newtheorem{theorem}{Theorem}
\newtheorem{corollary}[theorem]{Corollary}
\newtheorem{proposition}{Proposition}[section]
\newtheorem{lemma}[proposition]{Lemma}

\newtheorem{claim}[proposition]{Claim}
\newtheorem{cor}[proposition]{Corollary}
\theoremstyle{definition}

\newtheorem{example}{Example}

\newtheorem{remark}[proposition]{Remark}

\title[Characterizing the Cantor bi-cube]{Characterizing the Cantor bi-cube\\ in asymptotic categories}
\author{Taras Banakh and Ihor Zarichnyi}
\address{Instytut Matematyki, Uniwersytet Humanistyczno-Przyrodniczy Jana Kochanowskiego w Kielcach, Poland, and Department of Mathematics, Ivan Franko National University of Lviv, Ukraine}
\email{tbanakh@yahoo.com; T.O.Banakh@gmail.com; ihor.zarichnyj@gmail.com}
\begin{document}

\begin{abstract} We present the characterizations of metric spaces that are micro-, macro- or bi-uniformly equivalent to the extended Cantor set
$\EC=\big\{\sum_{i=-n}^\infty\frac{2x_i}{3^i}:n\in\IN ,\;(x_i)_{i\in\IZ}\in\{0,1\}^\IZ\big\}\subset\IR$, which is bi-uniformly equivalent to the Cantor bi-cube $2^{<\IZ}=\{(x_i)_{i\in\IZ}\in \{0,1\}^\IZ:\exists n\;\forall i\ge n\;x_i=0\}$ endowed with the metric $d((x_i),(y_i))=\max_{i\in\IZ}2^i|x_i-y_i|$.
The characterizations imply that any two (uncountable) proper isometrically homogeneous ultrametric spaces are coarsely (and bi-uniformly) equivalent. This implies that any two countable locally finite groups endowed with proper left-invariant metrics are coarsely equivalent. For the proof of these results we develop a technique of towers which can have an independent interest.
\end{abstract}

\maketitle


\section{Introduction}
This paper was motivated by the problem of coarse classification of countable locally finite groups posed in \cite{BDHM}, repeated in \cite[Problem 1606]{OP2}, and communicated to the authors by I.V.~Protasov. As we shall see later, a crucial role in this classification belongs to the extended Cantor set
$$\EC=\Big\{\sum_{i=-n}^\infty\frac{2x_i}{3^i}:n\in\IN ,\;(x_i)_{i\in\IZ}\in\{0,1\}^\IZ\Big\}\subset\IR.$$ So firstly we present  four characterizations of the extended Cantor set $\EC$ in various categories of metric spaces and then we shall apply these characterizations to the problem of coarse and bi-uniform classifications of locally finite groups (more generally of isometrically homogeneous metric spaces).

We shall mainly work in the categories of proper metric spaces and their (macro-, micro-, or bi-) uniform maps. It will be convenient to introduce such maps using the notion of the oscillation $\w_f$ of a function $f:X\to Y$  between metric spaces $X$ and $Y$. By definition, the {\em oscillation} of $f$ is the function $\w_f:[0,\infty)\to[0,\infty]$ assigning to each $\delta\ge 0$ the (finite or infinite) number $$\w_f(\delta)=\sup\{\dist(f(x),f(x')):x,x'\in X,\;\dist(x,x')\le\delta\}.$$
Here $\dist(x,x')$ denotes the distance between points $x,x'$ in a metric space.

A map $f:X\to Y$ is called
\begin{itemize}
\item {\em uniformly continuous} (or else {\em micro-uniform}\/) if for any $\forall \e>0\;\exists\delta>0$ with $\w_f(\delta)\le\e$;
\item {\em macro-uniform} if $\forall \delta<\infty\;\exists \e<\infty$ with $\w_f(\delta)\le\e$;
\item {\em bi-uniform} if $f$ is macro- and micro-uniform.
\end{itemize}

Those notions induce the corresponding equivalences of metric spaces. Namely, a map $f:X\to Y$ between two metric spaces is called
\begin{itemize}
\item a {\em uniform homeomorphism} if $f$ is bijective and both $f$ and $f^{-1}$ are uniformly continuous;
\item a {\em bi-uniform equivalence} if $f$ is bijective and both $f$ and $f^{-1}$ are bi-uniform maps;
\item a {\em coarse equivalence} if $f$ is macro-uniform and there exists a macro-uniform map $g:Y\to X$ such that $\dist(f\circ g,\id_Y)<\infty$ and $\dist(g\circ f,\id_X)<\infty$.
\end{itemize}

Observe that a map $f:X\to Y$ is a bi-uniform equivalence if and only if $f$ is both a uniform homeomorphism and a coarse equivalence.

We have defined morphisms and isomorphisms in our categories and now will switch to the objects.

 We shall say that a metric space $X$
\begin{itemize}
\item is {\em isometrically homogeneous} if for any two points $x,y\in X$ there is a bijective isometry  $f:X\to X$ such that $f(x)=y$;
\item is {\em proper} if $X$ is unbounded but for every $x_0\in X$ and $r\in[0,+\infty)$ the closed $r$-ball $B_r(x_0)=\{x\in X:\dist(x,x_0)\le r\}$ centered at $x_0$ is compact;
\item has {\em bounded geometry} if there is $\delta<\infty$ such that for every $\e<\infty$ there in $n\in\IN$ such that each $\e$-ball in $X$ can be covered by $\le n$ balls of radius $\delta$;
\item an {\em ultrametric space} if $d(x,y)\le\max\{d(x,z),d(z,y)\}$ for any points $x,y,z\in X$.
\end{itemize}

Ultrametric spaces often appear as natural examples of zero-dimensional spaces (in various senses), see \cite{BDHM}. We shall be interested in four notions of zero-dimensionality: topological, micro-uniform, macro-uniform (=asymptotic), and bi-uniform.

First, given a positive real number $s$ define the {\em $s$-connected component} of a point $x$ of a metric space $X$ as the set $C_s(x)$ of all points $y\in X$ that can be linked with $x$ by a chain of points $y=z_0,z_1,\dots,z_n=x$ such that $\dist(z_{i-1},z_i)\le s$ for all $i\le n$.
By $\C_s(X)=\{C_s(x):x\in X\}$ we denote the family of the (pairwise disjoint) $s$-connected components of $X$. Given a family $\C$ of subsets of a metric space $X$ let $$\mesh \C=\sup\limits_{C\in\C}\diam(C).$$

For a metric space $X$ and positive real numbers $\delta\le\e$ consider the following cardinal characteristics:
$$
\ent_\delta^\e(X)=\min_{x\in X}|C_\e(x)/\C_\delta(X)|\mbox{ \ and \ } \Ent_\delta^\e(X)=\sup_{x\in X}|C_\e(x)/\C_\delta(X)|\mbox{ \ where}
$$
$$|C_\e(x)/\C_\delta(X)|=|\{C_\delta(y):y\in C_\e(x)\}|$$is the number of
$\delta$-connected components composing the $\e$-connected component $C_\e(x)$ of $x$.

If the metric space $X$ is isometrically homogeneous, then $\theta_\delta^\e(X)=\Theta_\delta^\e(X)=|C_\e(x)/\C_\delta(X)|$ for every $x\in X$. If $X$ is an ultrametric space, then the $\e$-connected component $C_\e(x)$ of a point $x$ coincides with the closed $\e$-ball $B_\e(x)$ and thus $|C_\e(x)/\C_\delta(X)|$ is just the number of $\delta$-balls composing the $\e$-ball $B_\e(x)$. Observe that an ultrametric space $X$ has bounded geometry if and only if there is $\delta<\infty$ such that $\Theta_\delta^\e(X)$ if finite for every finite $\e\ge\delta$.

We shall say that a metric space $X$ has
\begin{itemize}
\item {\em topological dimension zero} if the family of closed-and-open subsets forms a base of the topology of $X$;
\item {\em micro-uniform dimension zero} if $\forall \e>0\;\exists\delta>0$ with $\mesh \C_\delta(X)\le\e$;
\item {\em macro-uniform} (or else {\em asymptotic}) {\em dimension zero}  if $\forall \delta<\infty\;\exists\e<\infty$ with $\mesh \C_\delta(X)\le\e$;
\item {\em bi-uniform dimension zero} if $X$ has both micro-uniform and macro-uniform dimensions zero.
\end{itemize} It follows that a metric space $X$ of bi-uniform dimension zero has  topological, micro-uniform, and macro-uniform dimensions zero.

If $X$ is an ultrametric space, then for every $s>0$ the $s$-connected component $C_s(x)$ of a point $x\in X$ coincides with the closed $s$-ball $B_s(x)$. So $X$ has bi-uniform dimension zero (because $\mesh\C_s(X)=s$ for all $s>0$). On the other hand, each metric space of asymptotic (bi-uniform) dimension zero is coarsely (bi-uniformly) equivalent to an ultrametric space, see Theorem 4.3 of \cite{BDHM}.

The class of proper metric spaces of bi-uniform dimension zero contains an interesting object
$$\EC =\Big\{\sum_{i=-n}^\infty\frac{2x_i}{3^i}:(x_i)_{i\in\IZ}\in\{0,1\}^\IZ,\;n\in\IN\Big\}\subset\IR$$
called the {\em extended Cantor set}. The extended Cantor set $\EC $ coincides with the image of the {\em Cantor bi-cube}
$$2^{<\IZ}=\{(x_i)_{i\in\IZ}\in\{0,1\}^\IZ: \exists n\in\IN\;\; \forall i>n\;\;\; (x_i=0)\}$$
under the map
$$f:2^{<\IZ}\to\EC,\;\;\; f:(x_i)_{i\in\IZ}\mapsto\sum_{i=-\infty}^\infty 2\cdot 3^i\cdot x_i.$$This map determines a bi-uniform equivalence between the extended Cantor set $\EC $ and the Cantor bi-cube $2^{<\IZ}$ endowed with the ultrametric
$$d((x_i),(y_i))=\max_{i\in\IZ}2^{i}|x_i-y_i|.$$
The Cantor bi-cube can be written as the product
$2^{<\IZ}=2^{\w}\times 2^{<\IN}$ of {\em the Cantor micro-cube}
$$2^\w=\{(x_i)_{i\in\IZ}\in 2^{<\IZ}:\mbox{$x_i=0$ for all $i>0$}\}$$ and the {\em  Cantor macro-cube}
$$2^{<\IN}=\{(x_i)_{i\in\IZ}\in 2^{<\IZ}:\mbox{$x_i=0$ for all $i\le 0$}\}.$$
The Cantor micro-cube can be identified with the standard Cantor cube $\{0,1\}^\w$.
It is well-known that the Cantor micro-cube $2^\w$ contains a micro-uniform copy of each zero-dimensional compact metric space \cite[7.8]{Ke}. The Cantor macro-cube $2^{<\IN}$ has a similar property: it contains a macro-uniform copy of each asymptotically zero-dimensional metric space of bounded geometry, see Theorem 3.11 of \cite{DZ}. This picture is completed by the following

\begin{theorem}[Universality of the Cantor bi-Cube]\label{univer} A metric space $X$ is bi-uniformly equivalent to a subspace of the Cantor bi-cube $2^{<\IZ}$ if and only if $X$ is a metric space of bi-uniform dimension zero such that $\Theta^\e_\delta(X)<\infty$ for all $0<\delta\le \e<\infty$.
\end{theorem}

Now we turn to the problem of characterization of the spaces $2^\w$, $2^{<\IN}$, and $2^{<\IZ}$ in various categories. The characterization of the Cantor micro-cube
is well-known and is due to Brouwer \cite[7.4]{Ke}:

\begin{theorem}[Topological Characterization of the Cantor Cube] For a metric space $X$ the following conditions are equivalent:
\begin{enumerate}
\item $X$ is topologically equivalent to the Cantor micro-cube $2^\w$;
\item $X$ is micro-uniformly equivalent to $2^\w$;
\item $X$ is bi-uniformly equivalent to $2^\w$;
\item $X$ is a zero-dimensional metric compact space without isolated points.
\end{enumerate}
\end{theorem}

Since the Cantor bi-cube $2^{<\IZ}=2^\w\times 2^{<\IN}$, and the Cantor macro-cube $2^{<\IN}$ is discrete, the preceding theorem implies the following (well-known) topological characterization of the Cantor bi-cube $2^{<\IZ}$:

\begin{theorem}[Topological Characterization of the Cantor bi-Cube]\label{top-char} A metric space $X$ is topologically equivalent to the Cantor bi-cube $2^{<\IZ}$ if and only if
\begin{enumerate}
\item $X$ has topological dimension zero;
\item $X$ is separable, locally compact and non-compact;
\item $X$ has no isolated points.
\end{enumerate}
\end{theorem}

In the next three theorems we present characterizations of the Cantor bi-cube in the micro-, macro-, and bi-uniform categories.

\begin{theorem}[Micro-Uniform Characterization of the Cantor bi-Cube]\label{mU-char} A metric space $X$ is micro-uniformly equivalent to the Cantor bi-cube $2^{<\IZ}$ if and only if
\begin{enumerate}
\item $X$ is a non-compact complete metric space of micro-uniform dimension zero;
\item there is $\e>0$ such that $\Theta^\e_\delta(X)$ is finite for all positive $\delta\le \e$ and $\lim\limits_{\delta\to+0}\theta^\e_\delta(X)=\infty$.
\end{enumerate}
\end{theorem}

 \begin{theorem}[Macro-Uniform Characterization of the Cantor bi-cube]\label{MU-char} A metric space $X$ is macro-uniformly equivalent to the Cantor bi-cube $2^{<\IZ}$ if and only if
\begin{enumerate}
\item $X$ has macro-uniform dimension zero;
\item there is $\delta>0$ such that $\Theta^\e_\delta(X)$ is finite for all positive $\e\ge \delta$ and $\lim\limits_{\e\to\infty}\theta^\e_\delta(X)=\infty$.
\end{enumerate}
\end{theorem}

\begin{theorem}[Bi-Uniform Characterization of the Cantor bi-Cube]\label{bU-char} A metric space $X$ is bi-uniformly equivalent to the Cantor bi-cube $2^{<\IZ}$ if and only if
\begin{enumerate}
\item $X$ is a complete metric space of bi-uniform dimension zero;
\item $\Theta^\e_\delta(X)$ is finite for all $0<\delta\le\e<\infty$;
\item $\lim\limits_{\e\to\infty}\theta^\e_\delta(X)=\infty$ for all $\delta<\infty$;
\item $\lim\limits_{\delta\to+0}\theta^\e_\delta(X)=\infty$ for all $\e>0$.
\end{enumerate}
\end{theorem}

It is clear that any metric space $X$ that is bi-uniformly equivalent to the Cantor bi-cube $2^{<\IZ}$ is micro-uniformly and macro-uniformly equivalent to $2^{<\IZ}$. The converse is not true.

\begin{example}\label{ex1} Let $\w$ be the space of finite ordinals, endowed with the discrete 2-valued metric. The metric space $2^\w\times\w\times 2^{<\IN}$ is micro-uniformly and macro-uniformly equivalent to $2^{<\IZ}$ but fails to be bi-uniformly equivalent to $2^{<\IZ}$.
\end{example}

Characterization Theorems~\ref{top-char}--\ref{bU-char} of the Cantor bi-cube allows us to detect copies of $2^{<\IZ}$ among isometrically homogeneous metric spaces:

\begin{corollary}\label{t4} An isometrically homogeneous metric space $X$ is
\begin{enumerate}
\item micro-uniformly equivalent to $2^{<\IZ}$ if and only if $X$ is homeomorphic to $2^{<\IZ}$ if and only if $X$ is uncountable, separable, locally compact, non-compact, and has topological dimension zero;
\item macro-uniformly equivalent to $2^{<\IZ}$ if and only if $X$ is unbounded, has bounded geometry and has asymptotic dimension zero;
\item bi-uniformly equivalent to $2^{<\IZ}$ if and only if $X$ is proper, uncountable, and has bi-uniform dimension zero.
\end{enumerate}
\end{corollary}

Now we apply this classification result to the macro- and bi-uniform classification of countable groups, viewed as metric spaces endowed with perfect left-invariant metrics. J.Smith \cite{Smith} observed that each countable group carries a perfect left-invariant metric and such a metric is unique up to the bi-uniform equivalence. A.Dranishnikov and J.Smith \cite{DS} proved that a countable group $G$ endowed with a proper left-invariant metric has asymptotic dimension zero if and only if $G$ is {\em locally finite} in the sense that each finitely-generated subgroup of $G$ is finite. The authors of \cite{BDHM} classified countable locally finite groups up to the bi-uniform equivalence and posed the problem of classification of countable locally finite groups up to the coarse equivalence. The same problem was repeated by J.Sanjurjo in \cite[Problem 1606]{OP2}. The following corollary of Corollary~\ref{t4}(2) answers this problem.

\begin{corollary}\label{c2} Any two countable locally finite groups endowed with proper left-invariant metrics are coarsely equivalent.
\end{corollary}

This corollary is a principal ingredient in the coarse classification  of countable abelian groups given in \cite{BHZ}.

Corollary~\ref{t4} shows that the coarse classification of proper isometrically homogeneous metric spaces of asymptotic dimension zero is trivial: all such spaces are coarsely equivalent. The same concerns the bi-uniform classification of {\em uncountable} proper isometrically homogeneous metric spaces of bi-uniform dimension zero: all such spaces are bi-uniformly equivalent.  Also the micro-uniform classification of countable proper isometrically homogeneous metric spaces is trivial: all such spaces are micro-uniformly equivalent to $\IZ$.
In contrast, the bi-uniform classification of {\em countable} proper isometrically homogeneous metric spaces of uniform dimension zero is non-trivial and yields continuum many non-equivalent spaces.

First observe that the Baire Theorem guarantees that each countable proper isometrically homogeneous metric space $X$ is {\em boundedly-finite} in the sense that all bounded subsets of $X$ are finite.

For each boundedly-finite metric space $X$ of asymptotic dimension zero we can consider the function $f_X:\Pi\to\w\cup\{\infty\}$ defined on the set $\Pi$ of prime numbers and assigning to each  $p\in\Pi$ the  number
$$f_X(p)=\sup\{n\in\w:\mbox{$p^n$ divides $|C_s(x)|$ for some $x\in X$ and $s>0$}\},$$
where $C_s(x)$ stands for the $s$-connected component of $x$. It turns out that the function $f_X$ completely determines the bi-uniform type of a countable proper isometrically homogeneous metric space $X$ of asymptotic dimension zero.

\begin{theorem}\label{t5} Two countable proper isometrically homogeneous metric spaces $X,Y$ of asymptotic dimension zero are bi-uniformly equivalent if and only if $f_X=f_Y$.
\end{theorem}

For countable groups (endowed with proper left-invariant metrics) Theorem~\ref{t5} has been proved in \cite{BDHM}.

Observe that for any function $f:\Pi\to\w\cup\{\infty\}$ there is a countable proper isometrically homogeneous ultrametric space $X$ with $f=f_X$. To get such a space $X$, consider the abelian group
$$\IZ_f=\bigoplus_{p\in\Pi}\IZ_p^{f(p)}.$$  If $f(p)=\infty$ then $\IZ_p^{f(p)}=\IZ_p^\infty$ is the direct sum of countably many copies of the cyclic group $\IZ_p=\IZ/p\IZ$.

Endowing the group $\IZ_f$ with a suitable proper left-invariant metric $d$, we can see that the metric space $X=(\IZ_f,d)$ has $f_X=f$. Combining this observation with Corollary~\ref{t4}(3) and Theorem~\ref{t5}, we get the following bi-uniform classification of proper isometrically homogeneous metric spaces of bi-uniform dimension zero.

\begin{corollary}\label{t6} A proper isometrically homogeneous metric space $X$ of bi-uniform dimension zero is bi-uniformly equivalent to
\begin{itemize}
\item the Cantor bi-cube $2^{<\IZ}$ if $X$ is uncountable;
\item the group $\IZ_{f_X}$ if $X$ is countable.
\end{itemize}
\end{corollary}

\section{Characterizing the coarse equivalence}\label{sa}

In this section we show that various natural ways of defining morphisms in Asymptology\footnote{The term ``Asymptology'' was coined by I.Protasov in \cite{PZ} for naming the theory studying large scale properties of metric spaces (or more general objects like {\em balleans} of I.~Protasov \cite{PZ}, \cite{PB} or {\em coarse structures} of J.~Roe \cite{Roe}).} lead to the same notion of coarse equivalence. Besides the original approach of J.~Roe \cite{Roe} based on the notion of a coarse map, we discuss an alternative approach based on the notion of a multi-map.

By a {\em multi-map} $\Phi:X\Ra Y$ between two sets $X,Y$ we understand any subset $\Phi\subset X\times Y$.

For a subset $A\subset X$ by $\Phi(A)=\{y\in Y:\exists a\in A\mbox{ with }(a,y)\in\Phi\}$ we denote the image of $A$ under the multi-map $\Phi$. Given a point $x\in X$ we write $\Phi(x)$ instead of $\Phi(\{x\})$.

The inverse $\Phi^{-1}:Y\Ra X$ to the multi-map $\Phi$ is the subset $\Phi^{-1}=\{(y,x)\in Y\times X: (x,y)\in\Phi\}\subset Y\times X$. For two multi-maps $\Phi:X\Ra Y$ and $\Psi:Y\Ra Z$ we define their composition $\Psi\circ\Phi:X\Ra Z$ as usual:
$$\Psi\circ\Phi=\{(x,z)\in X\times Z:\exists y\in Y\mbox{ such that $(x,y)\in \Phi$ and $(y,z)\in\Psi$}\}.$$

 A multi-map $\Phi$ is called {\em surjective} if $\Phi(X)=Y$ and {\em bijective} if $\Phi\subset X\times Y$ coincides with the graph of a bijective (single-valued) function.

The {\em oscillation} of a multi-map $\Phi:X\Ra Y$ between metric spaces is the function $\w_\Phi:[0,\infty)\to[0,\infty]$ assigning to each $\delta\ge 0$ the (finite or infinite) number $$\w_\Phi(\delta)=\sup\{\diam(\Phi(A)):A\subset X,\;\diam(A)\le\delta\}.$$ Observe that $\w_\Phi(\Phi)=0$ if and only if $\Phi$ is at most single-valued in the sense that $|\Phi(x)|\le 1$ for any $x\in X$.

A multi-map $\Phi:X\Ra Y$ between metric spaces $X$ and $Y$ is called
\begin{itemize}
\item {\em micro-uniform}  if $\forall\e>0\;\exists\delta>0$ with $\w_\Phi(\delta)\le\e$;
\item {\em macro-uniform} if
$\forall\delta<\infty\;\exists\e<\infty$ with $\w_\Phi(\delta)\le\e$;
\item {\em bi-uniform} if $\Phi$ is both micro-uniform and macro-uniform.
\end{itemize}

A multi-map $\Phi:X\Ra Y$  is called a {\em bi-uniform}  (resp. {\em micro-uniform}, {\em macro-uniform}) {\em embedding} if $\Phi^{-1}(Y)=X$ and both multi-maps $\Phi$ and $\Phi^{-1}$ are bi-uniform (resp.  micro-uniform, macro-uniform). If, in addition, $\Phi(X)=Y$, then $\Phi$ is called  a {\em bi-uniform} (resp. {\em micro-uniform}, {\em macro-uniform}) {\em equivalence}.

Two metric spaces $X,Y$ are called {\em bi-uniformly} (resp. {\em micro-uniformly}, {\em macro-uniformly})
{\em equivalent\/}  if there is a bi-uniform (resp. micro-uniform, macro-uniform) equivalence $\Phi:X\Ra Y$.

It follows that each micro-uniform multi-map is at most single-valued and thus is uniformly continuous in the usual sense. So, two metric spaces $X,Y$ are micro-uniformly equivalent if and only if they are uniformly homeomorphic. On the other hand, the notion of bi-uniform equivalence agrees with that given in the introduction.
In Proposition~\ref{p1} below we shall prove that metric spaces are macro-uniformly equivalent if and only if they are coarsely equivalent.

A subset $L$ of a metric space $X$ is called {\em large} if $B_r(L)=X$ for some $r\in\IR$, where $B_r(L)=\{x\in X:\dist(x,L)\le  r\}$ stands for the closed $r$-neighborhood of the set $L$ in $X$.

For two multi-maps $\Phi:\Psi:X\Ra Y$ between metric spaces let $$\dist(\Psi,\Phi)=\inf\{r\in[0,\infty]:\forall x\in X\;\;\Phi(x)\subset B_r(\Psi(x))\mbox{ and }\Psi(x)\subset B_r(\Phi(x))\}.$$

The following  characterization is the main (and unique) result of this section.

\begin{proposition}\label{p1} For metric spaces $X,Y$ the following assertions are equivalent:
\begin{enumerate}
\item $X$ and $Y$ are macro-uniformly equivalent;
\item $X$ and $Y$ are coarsely equivalent;
\item the spaces $X,Y$ contain bi-uniformly equivalent large subspaces $X'\subset X$ and $Y'\subset Y$;
\item there are two macro-uniform maps $f:X\to Y$, $g:Y\to X$ whose inverses $f^{-1}:Y\Ra X$ and $g^{-1}:X\Ra Y$ are macro-uniform and $\max\{\dist(g\circ f,\id_X),\dist(f\circ g,\id_Y)\}<\infty.$
\end{enumerate}
\end{proposition}

\begin{proof} To prove the equivalence of the items (1)--(4), it suffices to establish the implications $(1)\Ra(4)\Ra(2)\Ra(3)\Ra(1)$.
\smallskip

$(1)\Ra(4)$ Assuming that $X$ and $Y$ are macro-uniformly equivalent, fix a surjective macro-uniform multi-map $\Phi:X\Ra Y$ with surjective macro-uniform inverse $\Phi^{-1}:Y\Ra X$. Since the multi-map $\Phi^{-1}$ is surjective, for every $x\in X$ the subset $\Phi(x)\subset Y$ is not empty and thus contains some  point $f(x)\in\Phi(x)$.
It follows from the macro-uniformity of $\Phi$ that the map $f:X\to Y$ is macro-uniform. Since $f^{-1}(y)\subset \Phi^{-1}(y)$ for all $y\in Y$, the macro-uniformity $\Phi^{-1}$ implies the macro-uniformity of the multi-map $f^{-1}:Y\Ra X$.

By the same reason, the surjectivity of the multi-map $\Phi$ implies the existence of a map $g:Y\to X$ such that $g(y)\in\Phi^{-1}(y)$ for all $y\in Y$. The macro-uniformity $\Phi$ and $\Phi^{-1}$ implies that $g:Y\to X$ and $g^{-1}:X\Ra Y$ are macro-uniform.

Since the composition $\Phi^{-1}\circ\Phi:X\Ra X$ is macro-uniform, there is a constant $C<\infty$ such that $\diam \Phi^{-1}\circ \Phi(x)\le C$ for all $x\in X$. Observing that $\{x,g\circ f(x)\}\subset\Phi^{-1}\circ \Phi(x)$ we see that $\dist(g\circ f,\id_X)\le C<\infty$. By the same reason, $\dist(f\circ g,\id_Y)<\infty$.
\smallskip

The implication $(4)\Ra(2)$ trivially follows from the definition of the coarse equivalence given in the Introduction.
\smallskip

$(2)\Ra(3)$ Assume that there are two macro-uniform maps $f:X\to Y$, $g:Y\to X$ with $\dist(g\circ f,\id_X)<R$ and $\dist(f\circ g,\id_Y)< R$ for some real number $R$.
It follows that $B_R(f(X))=Y$ and hence the set $f(X)$ is large in $Y$.
Since $f$ is macro-uniform, the number $S=1+\w_f(1)$ is finite. Let $Y'\subset f(X)$ be a maximal $S$-separated subset of $f(X)$. The $S$-separated property of $Y'$ means that $\dist(y,y')\ge S$ for any distinct points $y,y'\in Y'$. The maximality of $Y'$ guarantees that $Y'$ is large in $f(X)$ and consequently, in $Y$.

Choose any subset $X'\subset X$ making the restriction $h=f|X':X'\to Y'$ bijective. Being a restriction of a macro-uniform map, the map $h$ is macro-uniform. The choice of the number $S$ guarantees that the set $X'$ is 1-separated and consequently, the map $h$ is micro-uniform. Since $Y'$ is $S$-separated the inverse map $h^{-1}:Y'\to X'$ is micro-uniform.

It remains to check that $h^{-1}$ is macro-uniform. Given arbitrary $\e<\infty$, use  the macro-uniformity of the map $g:Y\to X$ to conclude that the number $\delta=\w_g(\e)$ is finite.
Now take any points $y,y'\in Y'$ with $\dist(y,y')\le \e$ and let $x=h^{-1}(y)$ and $x'=h^{-1}(y')$. We claim that $\dist(x,x')\le\delta+2R$. By the choice of $\delta$,   $\dist(g\circ f(x),g\circ f(x'))=\dist(g(y),g(y'))\le\delta=\w_g(\e)$. Since $\dist(g\circ f,\id_X)\le R$, we conclude that
$$
\begin{aligned}
\dist(x,x')&\le\dist(x,g\circ f(x))+\dist(g\circ f(x),g\circ f(x'))+\dist(g\circ f(x'),x')\le\\
&\le R+\dist(g(y),g(y'))+R\le\delta+2R.
\end{aligned}
$$

Finally, let us show that the set $X'$ is large in $X$. Given any point $x\in X$, find a point $x'\in X'$ with $\dist(f(x),f(x'))\le S$. Then $\dist(x,x')\le \dist(x,g\circ f(x))+\dist(g\circ f(x),g\circ f(x'))+\dist(g\circ f(x'),x')\le R+\w_g(S)+R$ and consequently, $B_{R'}(X')=X$ for $R'=2R+\w_g(S)$.
\smallskip

$(3)\Ra(1)$ Assume that the spaces $X,Y$ contain bi-uniformly equivalent large subspaces $X'\subset X$ and $Y'\subset Y$ and let $f:X'\to Y'$ be a bi-uniform equivalence.
Find $R\in\IR$ such that $B_R(X')=X$ and $B_R(Y')=Y$. Take any surjective maps $\varphi:X\to X'$ and $\psi:Y\to Y'$ with $\dist(\varphi,\id_X)\le R$ and $\dist(\psi,\id_Y)\le R$. It is easy to see that $\varphi$ and $\psi$ are macro-uniform equivalences and then the composition $\psi^{-1}\circ f\circ \varphi:X\Ra Y$ is a required macro-uniform equivalence between $X$ and $Y$.
\end{proof}

\section{$\e$-Connected components and uniform multi-maps}

We recall that for $\e>0$ and a point $x$ of a metric space $X$ by $C_\e(x)$ we denote the $\e$-connected component of $x$. This is the set of all points $x'\in X$ that can be linked with $x$ by a chain of points $x=x_0,x_1,\dots,x_n=x'$ with $\dist(x_{i-1},x_i)\le\e$ for all $i\le n$. By $\C_\e(X)=\{C_\e(x):x\in X\}$ we denote the family of all $\e$-connected components of $X$.

\begin{lemma}\label{l1} Let $\Phi:X\Ra Y$ be a multi-map such that $\Phi^{-1}(Y)=X$.  For any real numbers $\delta\ge 0$ and $\e\ge \w_\Phi(\delta)$, and  every point $x\in X$ the image $\Phi(C_\delta(x))$ lies in the $\e$-connected component $C_\e(y)$ of any point $y\in\Phi(x)$.
\end{lemma}

\begin{proof} Given any $x'\in C_\delta(x)$ and $y'\in\Phi(x)$, we need to check that $y'\in C_\e(y)$. Find a chain of points $x=x_0,x_1,\dots,x_n=x'$ such that $\dist(x_{i-1},x_i)\le\delta$ for all $i\le n$.
Since $X=\Phi^{-1}(Y)$, for every $i\le n$ we can choose a point $y_i\in\Phi(x_i)$ so that $y_0=y$ and $y_n=y'$. It follows from the definition of $\w_\Phi(\delta)$ that
for every $i\le n$ we get
$$\dist(y_{i-1},y_i)\le\diam \Phi(\{x_{i-1},x_i\})\le\w_\Phi(\dist(x_{i-1},x_i))\le\w_\Phi(\delta)\le\e,$$
which means that $y=y_0,y_1,\dots,y_n=y'
$ is an $\e$-chain linking the points $y$ and $y'$. Consequently, $y'\in C_\e(y)$.
\end{proof}

Lemma~\ref{l1} will be applied in order to show that some information on the asymptotic properties of the cardinal numbers $\theta^\e_\delta(X)$ and $\Theta_\delta^\e(X)$ is preserved by bi-uniform equivalences.

\begin{lemma}\label{l2} Let $\Phi:X\Ra Y$ is a multi-map such that $Y=\Phi(X)$ and $\Phi^{-1}(Y)=X$. For any positive real numbers $\delta<\e$ and $\delta'<\e'$ with $\e'\ge\w_\Phi(\e)$, $\delta\ge\w_{\Phi^{-1}}(\delta')$ we get $\theta_{\delta}^\e(X)\le\theta_{\delta'}^{\e'}(Y)$ and $\Theta_{\delta}^\e(X)\le\Theta_{\delta'}^{\e'}(Y)$.
\end{lemma}

\begin{proof} For any $\delta$-connected component $C\in\C_\delta(X)$ choose a point $y_C\in\Phi(C)$. Since $\w_{\Phi^{-1}}(\delta')\le\delta$, we can apply Lemma~\ref{l1} to prove that for any distinct components $C,C'\in C_\delta(X)$ the points $y_C$ and $y'_C$ lie in distinct $\delta'$-components of $Y$. Therefore the map $$\varphi:\C_\delta(X)\to\C_{\delta'}(Y),\;\;\varphi:C\mapsto C_{\delta'}(y_C)$$ is injective.

By Lemma~\ref{l1}, for any point $x\in X$ the set $\Phi(C_\e(x))$ lies in $C_{\e'}(y)$ for any $y\in\Phi(x)$. Now the injectivity of the map $\varphi$ implies that
$$|C_\e(x)/\C_\delta(X)|\le |C_{\e'}(y)/\C_{\delta'}(Y)|\le \Theta_{\delta'}^{\e'}(Y)$$ and hence $\Theta_{\delta}^\e(X)\le\Theta_{\delta'}^{\e'}(Y)$.

Next, find a point $y\in Y$ with $\theta_{\delta'}^{\e'}(Y)=|C_{\e'}(y)/\C_{\delta'}(Y)|$ and choose any point $x\in\Phi^{-1}(y)$. Then
$$\theta_\delta^\e(X)\le|C_\e(x)/\C_\delta(X)|\le |C_{\e'}(y)/\C_{\delta'}(Y)|=\theta_{\delta'}^{\e'}(Y).$$
\end{proof}

\section{Towers}

The Characterization Theorems announced in the introduction will be proved by induction on partially ordered sets called towers. A typical example of a tower is the set $\{B_{2^n}(x):x\in X,\; n\in\IZ\}$ of closed $2^n$-balls of an ultrametric space $X$, ordered by the inclusion relation. To give a precise definition of a tower we need to recall some standard notions related to partially ordered sets.

\subsection{Partially ordered sets} A {\em partially ordered set} is a set $T$ endowed with a reflexive antisymmetric transitive relation $\le$.

A partially ordered set $T$ is called {\em $\upa$-directed} (resp. {\em $\da$-directed\/}) if for any two points $x,y\in T$ there is a point $z\in T$ such that $z\ge x$ and $z\ge y$ (resp. $z\le x$ and $z\le y$).

A subset $C$ of a partially ordered set $T$ is called {\em $\da$-cofinal} (resp. {\em $\upa$-cofinal}) if for every $x\in T$ there is $y\in C$ such that $y\le x$ (resp. $y\ge x$); $C$ is called {\em $\updownarrow$-cofinal} in $T$ if it is $\da$-cofinal and $\upa$-cofinal in $T$.

By the {\em lower cone} (resp. {\em upper cone}) of a point $x\in T$  we understand  the set ${\downarrow}x=\{y\in T:y\le x\}$ (resp. ${\uparrow}x=\{y\in T:y\ge x\}$). A subset $A\subset T$ will be called a {\em lower} (resp. {\em upper}) {\em set} if ${\downarrow}a\subset A$ (resp. ${\uparrow}a\subset A$) for all $a\in A$.
For two points $x\le y$ of $T$ the intersection $[x,y]={\upa}x\cap {\da}y$ is called the {\em order interval} with end-points $x,y$.

A partially ordered set $T$ is a {\em tree} if $T$ is $\da$-directed and for each point $x\in T$ the lower cone ${\da}x$ is well-ordered (in the sense that each subset $A\subset{\da}x$ has the smallest element.

\subsection{Introducing towers}

A partially ordered set $T$ is called a {\em tower} if
$T$ is $\upa$-directed and for every points $x\le y$ in $T$ the order interval $[x,y]\subset T$ is finite and linearly ordered.

This definition implies that for every point $x$ in a tower $T$ the upper set ${\upa}x$ is linearly ordered and is order isomorphic to a subset of $\w$. Since $T$ is $\upa$-directed, for any points $x,y\in T$ the upper sets ${\upa}x$ and ${\upa}y$ have non-empty intersection and this intersection has the smallest element $x\wedge y=\min({\upa}x\cap{\upa}y)$ (because each order interval in $X$ is finite). Thus any two points $x,y$ in a tower have the smallest upper bound $x\wedge y$.

It follows that for each point $x\in T$ of a tower the lower cone ${\da}x$ endowed with the reverse partial order is a tree of at most countable height.

\subsection{Levels of a tower} The definition of a tower $T$ includes the condition that  for any points $x\le y$ of  $T$ the order interval $[x,y]={\upa}x\cap {\da}y$ is linearly ordered and finite. This allows us to define levels of the tower $T$ as follows.

Given two points $x,y\in T$ we write $\lev_T(x)\le\lev_T(y)$ if $$|[x,x\wedge y]|\ge|[y,x\wedge y]|.$$ Also we write $\lev_T(x)=\lev_T(y)$ if  $|[x,x\wedge y]|=|[y,x\wedge y]|$.

The relation $$\{(x,y)\in T\times T:\lev_T(x)=\lev_T(y)\}$$ is an equivalence relation  on $T$ dividing the tower $T$ into equivalence classes called the {\em levels} of $T$. The level containing a point $x\in T$ is denoted by $\lev_T(x)$.
Let $$\Lev(T)=\{\lev_T(x):x\in T\}$$ denote the set of levels of $T$ and
$$\lev_T:T\to\Lev(T),\;\lev_T:x\mapsto\lev_T(x),$$
stand for the quotient map called the {\em level map}. If the tower $T$ is clear from the context, we shall omit the subscript $T$ and write $\lev$ instead of $\lev_T$.

The set $\Lev(T)$ of levels of $T$ endowed with the order $\lev_T(x)\le \lev_T(y)$ is a linearly ordered set, order isomorphic to a subset of integers.
For a level $\lambda\in\Lev(T)$ by $\lambda+1$ (resp. $\lambda-1$) we denote the successor (resp. the predecessor) of $\lambda$ in the level set $\Lev(T)$. If $\lambda$ is a maximal (resp. minimal) level of $T$, then we put $\lambda+1=\emptyset$ (resp. $\lambda-1=\emptyset$).

An embedding of the level set $\Lev(T)$ into $\IZ$ can be constructed as follows. Pick any point $\theta\in T$ and consider the map $e_\theta:\Lev(T)\to\IZ$ assigning to each level $\lev_T(x)\in\Lev(T)$ the integer number $$|[x,x\wedge \theta]|-|[\theta,x\wedge \theta]|.$$ In such a way we label the levels of $T$ by integer numbers so that the point $\theta$ sits on the zeros level.

The following model of the famous Eiffel tower is just an example of a tower having seven levels.

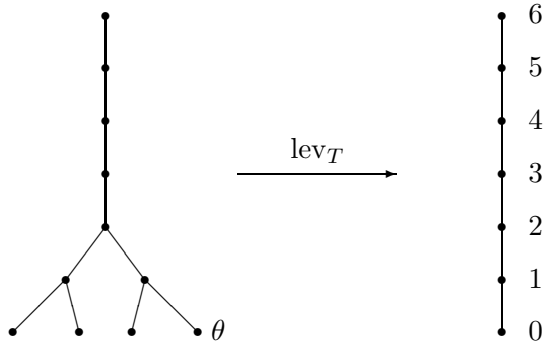
\begin{picture}(100,145)(-150,-10)

\put(0,120){\circle*{3}}
\put(0,120){\line(0,-1){20}}
\put(0,100){\circle*{3}}
\put(0,100){\line(0,-1){20}}
\put(0,80){\circle*{3}}
\put(0,80){\line(0,-1){20}}
\put(0,60){\circle*{3}}
\put(0,60){\line(0,-1){20}}
\put(0,40){\circle*{3}}
\put(0,40){\line(-3,-4){15}}
\put(0,40){\line(3,-4){15}}
\put(-15,20){\circle*{3}}
\put(-15,20){\line(-1,-1){20}}
\put(-35,0){\circle*{3}}
\put(15,20){\line(1,-1){20}}
\put(15,20){\circle*{3}}
\put(-15,20){\line(1,-4){5}}
\put(15,20){\line(-1,-4){5}}
\put(-10,0){\circle*{3}}
\put(10,0){\circle*{3}}
\put(35,0){\circle*{3}}
\put(40,-3){$\theta$}

\put(50,60){\vector(1,0){60}}
\put(70,66){$\lev_T$}

\put(150,120){\circle*{3}}
\put(160,117){6}
\put(150,120){\line(0,-1){20}}
\put(150,100){\circle*{3}}
\put(160,97){5}
\put(150,100){\line(0,-1){20}}
\put(150,80){\circle*{3}}
\put(160,77){4}
\put(150,80){\line(0,-1){20}}
\put(150,60){\circle*{3}}
\put(160,57){3}
\put(150,60){\line(0,-1){20}}
\put(150,40){\circle*{3}}
\put(160,37){2}
\put(150,40){\line(0,-1){20}}
\put(150,20){\circle*{3}}
\put(160,17){1}
\put(150,20){\line(0,-1){20}}
\put(150,0){\circle*{3}}
\put(160,-3){0}
\end{picture}

A tower $T$ will be called {\em $\da$-bounded} (resp. {\em $\upa$-bounded}) if the level set $\Lev(T)$ has the smallest (resp. largest) element. Otherwise $T$ is called {\em $\da$-unbounded} (resp. {\em $\upa$-unbounded}). A tower $T$ is called {\em $\updownarrow$-unbounded} if it is $\da$-unbounded and $\upa$-unbounded. Let us observe that $\upa$-bounded towers endowed with the reverse partial order are trees of at most countable height.

\subsection{A tower induced by a decomposition of a group} Let $G$ be a group written as the countable union $G=\bigcup_{n\in\w}G_n$ of a strictly increasing sequence
$$\{e\}=G_0\subset G_1\subset\cdots
$$
of subgroups of $G$.

Consider the family of cosets $T=\{xG_n:x\in G,\;n\in\w\}$ partially ordered by the inclusion relation. It is easy to check that the partially order set $T$ is a tower. This tower is $\da$-bounded and $\upa$-unbounded. For every $n\in\w$ the family of cosets $\{xG_n:x\in G\}$ forms a level of $T$. The minimal level of $G$ consists of the singletons and hence can be identified with the whole group $G$.

\subsection{The boundary of a tower}
By a {\em branch} of a tower $T$ we understand a maximal linearly ordered subset of $T$. The family of all branches of $T$ is denoted by $\partial T$ and is called the {\em boundary} of $T$. The boundary $\partial T$ carries an ultrametric that can be defined as follows.

Let $f:\Lev(T)\to[0,\infty)$ be a strictly increasing function such that
\begin{itemize}
\item $\inf f(\Lev(T))=0$ if $T$ is $\da$-unbounded and
\item $\sup f(\Lev(T))=\infty$ if $T$ is $\upa$-unbounded.
\end{itemize}
Such a map $f$ will be called a {\em scaling function} on $\Lev(T)$.

Given two branches $x,y\in\partial T$ let
$$\rho_f(x,y)=\begin{cases}0,&\mbox{if $x=y$,}\\
f\big(\lev_T(\min x\cap y)\big),&\mbox{if $x\ne y$.}
\end{cases}
$$
It is a standard exercise to check that $\rho_f$ is a well-defined ultrametric on the boundary $\partial T$ of $T$ turning $\partial T$ into a complete ultrametric space.
The following easy proposition says that the bi-uniform structure on $\partial T$ induced by the ultrametric $\rho_f$ does not depend on the choice of a scaling function $f$.

\begin{proposition}\label{p2} For any two scaling functions $f,g:\Lev(T)\to(0,\infty)$ the identity map $\id:(\partial T,\rho_f)\to(\partial T,\rho_g)$ is a bi-uniform equivalence.
\end{proposition}

In the sequel we shall assume that the boundary $\partial T$ of any tower $T$ is endowed with the ultrametric $\rho_f$ induced by some scaling function $f:\Lev(T)\to(0,\infty)$.

\subsection{Degrees of points of a tower}

For a point $x\in T$ and a level $\lambda\in\Lev(T)$ let $\suc_\lambda(x)=\lambda\cap{\downarrow}x$ be the set of predecessors of $x$ on the $\lambda$-th level and $\deg_\lambda(x)=|\suc_\lambda(x)|$. For $\lambda=\lev_T(x)-1$,  the set $\suc_{\lambda}(x)$, called the set of parents of $x$, is denoted by $\suc(x)$. The cardinality $|\suc(x)|$ is called the {\em degree} of $x$ and is denoted by $\deg(x)$. Thus $\deg(x)=\deg_{\lev_T(x)-1}(x)$. It follows that $\deg(x)=0$ if and only if $x$ is a minimal element of $T$.

For levels $\lambda,l\in\Lev(T)$ let
$$\deg_\lambda^l(T)=\min\{\deg_\lambda(x):\lev_T(x)=l\}\mbox{ and }\Deg_\lambda^l(T)=\sup\{\deg_\lambda(x):\lev_T(x)=l\}.$$

Now let us introduce several notions related to degrees. We define a tower $T$ to be
\begin{itemize}
\item {\em homogeneous} if $\deg_\lambda^\ell(T)=\Deg_\lambda^\ell(T)$ for any level $\lambda\le\ell$ of $T$;
\item {\em pruned} if $\deg^{\lambda+1}_\lambda(T)>0$ for every non-maximal level $\lambda$ of $T$;
\item {\em ${\upa}$-branching} if $\forall\lambda\in\Lev(T)\;\;\exists l\in\Lev(T)$ with $\Deg_\lambda^l(T)>1$;
\item {\em ${\da}$-branching} if $\forall \lambda\in\Lev(T)\;\;\exists l\in\Lev(T)$ with $\deg^\lambda_l(T)>1$;
\item {\em ${\dupa}$-branching} if $T$ is both $\da$-branching and $\upa$-branching.
\end{itemize}

It is easy to check that a tower $T$ is pruned if and only if each branch of $T$ meets each level of $T$. A tower $T$ is $\upa$-branching if no level $\lambda\in\Lev(T)$ has an upper bound in $T$.

By a {\em binary tower} we understand an ${\upa}$-unbounded homogeneous tower $T$ such that  $\deg_\lambda^{\lambda+1}(T)=2$ for each non-maximal level $\lambda$ of $T$. It is clear that each binary tower is pruned and $\upa$-branching.

\begin{remark} The ex-Cantor set $2^{<\IZ}$ (resp. anti-Cantor set $2^{<\IN}$) can be identified with the boundary $\partial T_2$ of a ${\da}$-unbounded (resp. $\da$-bounded) binary tower $T_2$.
\end{remark}

There is a direct dependence between the degrees of points of the tower $T$ and the capacities of the balls in the ultrametric space $\partial T$. We recall that for positive real numbers $\delta\le \e$ and a point $x\in X$ by $|C_\e(x)/\C_\delta(X)|$ we denote the cardinality of the set $\{C_\delta(y):y\in C_\e(x)\}$ of $\delta$-connected components of $X$ that lie in the $\e$-connected component of $y$ in $X$. If $X$ is an ultrametric space then $C_\e(x)/\C_\delta(X)$ is equal to the number of $\delta$-balls composing the $\e$-ball $B_\e(x)$.

\begin{proposition}\label{p3} Let be a tower and $f:\Lev(T)\to(0,\infty)$ be a scaling function determining the ultrametric $\rho_f$ on the boundary $\partial T$ of $T$. For any branch $\beta\in\partial T$, a point $x\in \beta$ with $n=\lev_T(x)$, and a level $k\le n$ of $T$ we get $\deg_{k}(x)=|C_{f(n)}(\beta)/\C_{f(k)}(\partial T)|$. Consequently,
$$\deg_k^n(T)=\theta_{f(k)}^{f(n)}(\partial T)\mbox{ and }\Deg_k^n(T)=\Theta_{f(k)}^{f(n)}(\partial T).$$
\end{proposition}

The proof is easy and is left to the reader as an exercise.

\subsection{Assigning a tower to a metric space} In the preceding section to each tower $T$ we have assigned the ultrametric space $\partial T$. In this section we describe the converse operation assigning to each metric space $X$ a pruned tower $T_X^L$ whose boundary $\partial T_X^L$ is canonically related to the space $X$.

A subset $L\subset[0,\infty)$ is called a {\em level set} if
\begin{itemize}
\item $\sup L=\infty$ and hence $L$ is $\upa$-cofinal in $[0,\infty)$;
\item $L$ is a tower in the sense that $[x,y]\cap L$ is finite for all $x,y\in L$;
\item $\inf L=0$ if $L\cap(-\infty,x]$ is infinite for some $x\in L$.
\end{itemize}
A level set $L\subset[0,\infty)$ is called {\em ${\da}$-bounded} if it has the smallest element. Otherwise, $L$ is $\da$-unbounded.

Given a metric space $X$ and a level set $L\subset[0,\infty)$ consider the set $$T^L_X=\{(C_\lambda(x),\lambda):x\in X,\,\lambda\in L\}$$endowed with the partial order $(C_\lambda(x),\lambda)\le (C_l(y),l)$ if $\lambda\le l$ and $C_\lambda(x)\subset C_l(y)$. Here, as expected, $C_\lambda(x)$ stands for the $\lambda$-connected component of $x$ in $X$.

\begin{proposition}\label{p4} The
partially ordered set $T^L_X$ is a pruned tower whose level set $\Lev(T_X^L)$ can be identified with $L$. If the metric space $X$ is isometrically homogeneous, then the tower $T_X^L$ is homogeneous.
\end{proposition}

\begin{proof}
To see that the partially ordered set $T^L_X$ is $\upa$-directed, take two elements $(C_\alpha(x),\alpha),(C_\beta(y),\beta)\in T_X^L$ and find a number $\lambda\in L$ such that $\lambda\ge\max\{\alpha,\beta,\dist(x,y)\}$ (such a number $\lambda$ exists because $\sup L=\infty$). Then $(C_\lambda(x),\lambda)$ is an upper bound for $(C_\alpha(x),\alpha)$ and $(C_\beta(y),\beta)$ in $T^L_X$.

Next, given two points $u=(C_\alpha(x),\alpha)$, $v=(C_\beta(y),\beta)$ in $T_X^L$ with $u\le v$, we need to check that the order interval $[u,v]$ is linearly ordered and finite.
Take any two points $t_1,t_2\in[u,v]$ and for every $i\in\{1,2\}$ find a point $z_i\in X$ and a real number $\lambda_i\in L$ such that $t_i=(C_{\lambda_i}(z_i),\lambda_i)$. It follows from $u\le t_i\le v$ that $\alpha\le \lambda_i\le \beta$ and $C_\alpha(x)\subset C_{\lambda_i}(z_i)\subset C_\beta(y)$.

Without loss of generality, $\lambda_1\le \lambda_2$. Since $C_\alpha(x)\subset C_{\lambda_2}(z_1)\cap C_{\lambda_2}(z_2)$, the $\lambda_2$-connected components $
C_{\lambda_2}(z_1),C_{\lambda_2}(z_2)$ coincide and hence $C_{\lambda_1}(z_1)\subset C_{\lambda_2}(z_1)=C_{\lambda_2}(z_2)$. Thus $t_1\le t_2$, witnessing that $[u,v]$ is linearly ordered.

By the same reason $\lambda_1=\lambda_2$ implies $t_1=t_2$. This ensures that the projection
$$\pr:[u,v]\to [\alpha,\beta]\cap L,\;\;\pr:(C_\lambda(z),\lambda)\mapsto \lambda,$$
is bijective and hence $|[u,v]|\le|[\alpha,\beta]\cap L|$ is finite.

It follows that the projection $$\pr:T_X^L\to L,\;\;\pr:(C_\lambda(x),\lambda)\mapsto\lambda$$ is a monotone surjective level-preserving map and for every $\lambda\in L$ the preimage $\pr^{-1}(\lambda)=\{(C_\lambda(x),\lambda):x\in X\}$ coincides with a level of the tower $T_X^L$. So, the set $L$ can be identified with the set $\Lev(T^L_X)$ of levels of the tower $T_X^L$.

To see that the tower $T^L_X$ is pruned, take any point $t=(C_\lambda(x),\lambda)\in T_X$ on a non-minimal level $\lambda\in L$ and let $\lambda^-\in L$ be the predecessor of $\lambda$ in $L$. Then the element $(C_{\lambda^-}(x),\lambda^-)$ is a parent of $t$, witnessing that $\deg(t)>0$ and $T_X$ is pruned.
\smallskip

If the metric space $X$ is isometrically homogeneous, then the tower $T^L_X$ is homogeneous because
for each point $t=(C_\lambda(x),\lambda)\in T^L_X$ and each level $\ell\in L$, $\ell\le\lambda$, the degree $\deg_\ell(t)=|C_\lambda(x)/\C_\ell(X)|$ does not depend on the point $x$. So, $\deg^\lambda_\ell(T^L_X)=\Deg^\lambda_\ell(T^L_X)$, witnessing the homogeneity of the tower $T^L_X$.
\end{proof}

The tower $T^L_X$ will be called the {\em canonical $L$-tower} of a metric space $X$.
The boundary $\partial T^L_X$ is endowed with the ultrametric $\rho_\id$ induced by the identity scaling function $\id:L\to [0,\infty)$. This ultrametric on $\partial T_X^L$ will be called {\em canonical}.

Observe that for each point $x\in X$ the set $C_L(x)=\{(C_\lambda(x),\lambda):\lambda\in L\}$ is a branch of the tower, so the map
$$C_L:X\to\partial T^L_X,\;\;C_L:x\mapsto C_L(x),$$ called the {\em canonical map}, is well-defined.

\begin{proposition} \label{p5}
\begin{enumerate}
\item $\dist(C_L(x),C_L(y))\le \inf\{\lambda\in L:\lambda\ge d(x,y)\}$ for all $x,y\in X$.
\item The canonical map $C_L:X\to\partial T^L_X$ is macro-uniform.
\item If $0\notin L$, then the canonical map $C_L$ is micro-uniform.
\item If $L$ is $\da$-bounded, then the canonical map $C_L$ is surjective.
\item The canonical map $C_L$ has dense image $C_L(X)$ in $\partial T^L_X$.
\item The inverse multi-map $C_L^{-1}:\partial T_X^L\Ra X$ is macro-uniform if and only if $X$ has macro-uniform dimension zero.
\item If $L$ is $\da$-unbounded, then the inverse multi-map $C_L^{-1}:\partial T_X^L\Ra X$ is micro-uniform if and only if $X$ has micro-uniform dimension zero.
\end{enumerate}
\end{proposition}

\begin{proof} 1. Given any two points $x,y\in X$ let $\lambda=\inf \big(L\cap[\dist(x,y),\infty)\big)$ and observe that $C_\lambda(x)=C_\lambda(y)$, which implies that $\dist(C_L(x),C_L(y))\le\lambda$.
\smallskip

2. The preceding item implies immediately that the canonical map $C_L:X\to\partial T^L_X$ is macro-uniform.
\smallskip

3. Assume that $0\notin L$. If $\inf L>0$, then for any positive $\delta<\inf L$ we get $\w_{C_L}(\delta)=0$ and thus $C_L$ is micro-uniform.

If $\inf L=0$, then for every $\e>0$ we can find $\delta\in L\cap(0,\e]$ and observe that $\w_{C_L}(\delta)=\delta\le \e$, witnessing that $C_L$ is micro-uniform.
\smallskip

4. If $L$ is $\da$-bounded, then $L$ has a  minimal element $\lambda_0$. It follows that each branch $\beta$ of the tower $T^L_X$ is equal to $C_L(x)$ for a point $x\in X$ whose $\lambda_0$-connected component $C_{\lambda_0}(x)$ coincides with the smallest element of the branch $\beta$. In this case the map $C_L$ is surjective.
\smallskip

5. If $L$ is $\da$-bounded, then the map $C_L$ is surjective by the preceding item and hence has dense image $C_L(X)$ in $\partial T^L_X$.

If $L$ is $\da$-unbounded, then
$\inf L=0\notin L$. Given any branch $\beta\in\partial T^L_X$ and any $\e>0$, we can find $\lambda\in L\cap(0,\e)$ and a point $x\in X$ with $(C_\lambda(x),\lambda)\in\beta$. Then $\dist(\beta,C_L(x))\le\lambda<\e$, witnessing that the image $C_L(X)$ is dense in $\partial T^L_X$.
\smallskip

6. Assume that the inverse multi-map $C_L^{-1}:\partial T^L_X\Ra X$ is macro-uniform.
To show that $X$ has macro-uniform dimension zero, we need to show that $\mesh \C_\delta(X)$ is finite for every $\delta<\infty$. Find any $\lambda\in L\cap[\delta,\infty)$ and put $\e=\w_{C_L^{-1}}(\lambda)$.

We claim that $\mesh \C_\delta(X)\le\e$. Indeed, given any $\delta$-connected component $C\in\C_\delta(X)$ and any points $x,y\in C$ we get $\dist(C_L(x),C_L(y))\le\lambda$ and $\dist(x,y)\le\diam C_L^{-1}(\{C_L(x),C_L(y)\})\le\w_{C_L^{-1}}(\lambda)\le\e$. Then $\diam C\le\e$ and $\mesh\C_\delta(X)\le\e$, witnessing that the metric space $X$ has macro-uniform dimension zero.

Now assume conversely that $X$ has macro-uniform dimension zero. In order to show that the inverse multi-map $C_L^{-1}:\partial T_X^L\Ra X$ is macro-uniform, given any $\delta<\infty$ find $\lambda\in L\cap[\delta,\infty)$ and put $\e=\mesh \C_\lambda(X)$. The number $\e$ is finite because $X$ has macro-uniform dimension zero. We claim that $\w_{C_L^{-1}}(\delta)\le\e$. Take any subset $A\subset\partial T_X^L$ with $\diam A\le\delta$. We need to show that $\diam C_L^{-1}(A)\le\e$. Take any points $x,y\in C_L^{-1}(A)$ and observe that $C_L(x),C_L(y)\in A$. Since $\dist(C_L(x),C_L(y))\le\delta\le\lambda$, $C_\lambda(x)=C_\lambda(y)$ and then $\dist(x,y)\le\mesh \C_\lambda(X)=\e$ and hence $\diam A\le\e$.
\smallskip

7. Assume that $L$ is $\da$-unbounded. If $X$ has micro-uniform dimension zero, then for any $\e>0$ we can find $\lambda\in L\cap(0,\e)$ and take $\delta>0$ so small that $\mesh \C_\delta(X)\le\lambda$. Repeating the argument from the preceding item, we can prove that $\w_{C_L^{-1}}(\delta)\le\lambda\le\e$, witnessing that $C_L^{-1}$ is micro-uniform.

Finally assume that $C_L^{-1}$ is micro-uniform. Then for every $\e>0$ we can find $\delta\in L$ with $\w_{C_L^{-1}}(\delta)\le\e$. Repeating the argument from the proof of the preceding item, we can check that $\mesh \C_\delta(X)\le\e$, witnessing that $X$ has micro-uniform dimension zero.
\end{proof}

The statements (2), (3), (6), (7) of Proposition~\ref{p5} imply:

\begin{cor}\label{c4} Let $L\subset[0,\infty)$ be a level set. The canonical map $C_L:X\to\partial T_X^L$ of a metric space $X$ into the boundary of its canonical $L$-tower $T_X^L$ is:
\begin{enumerate}
\item a macro-uniform embedding if and only if $X$ has macro-uniform dimension zero;
\item a micro-uniform embedding (if and) only if $X$ has micro-uniform dimension zero (and $L$ is $\da$-unbounded);
\item a bi-uniform embedding (if and) only if $X$ has bi-uniform dimension zero (and $L$ is $\da$-unbounded).
\end{enumerate}
\end{cor}

Combining this corollary with Proposition~\ref{p5}(4,5) we get another

\begin{cor}\label{c5} Let $L\subset[0,\infty)$ be a level set. The canonical map $C_L:X\to\partial T_X^L$ of a metric space $X$ into the boundary of its canonical $L$-tower is:
\begin{enumerate}
\item a macro-uniform equivalence (if and) only if $X$ has macro-uniform dimension zero (and $L$ is $\da$-bounded);
\item a micro-uniform equivalence (if and) only if $X$ is a complete metric space of  micro-uniform dimension zero (and $L$ is $\da$-unbounded);
\item a bi-uniform equivalence (if and) only if $X$ is a complete metric space of  bi-uniform dimension zero (and $L$ is $\da$-unbounded).
\end{enumerate}
\end{cor}

\begin{proof} 1. The first item is a direct corollary of Corollary~\ref{c4}(1) and Proposition~\ref{p5}(4).

2. If $C_L:X\to\partial T^L_X$ is a micro-uniform equivalence (that, is a uniform homeomorphism), then the metric space $X$ is complete because so is the ultrametric space $\partial T^L_X$. Corollary~\ref{c4}(2) implies that $X$ has micro-uniform dimension zero.

Now assume conversely that the metric space $X$ is complete and has micro-uniform dimension zero and the level set $L$ is $\da$-unbounded. By Corollary~\ref{c4}(2), the canonical map $C_L:X\to\partial T^L_X$ is a micro-uniform embedding and by Proposition~\ref{p5}(5), the image $C_L(X)$ is dense in $\partial T^L_X$. The metric space $C_L(X)\subset\partial T^L_X$, being uniformly homeomorphic to the complete metric space $X$, is complete and hence coincides with $\partial T^L_X$. Then the canonical map $C_L$, being a surjective micro-uniform embedding, is a micro-uniform equivalence.

3. The third statement can be proved by analogy with the second one.
\end{proof}

\begin{remark} By its spirit the correspondence between towers and metric spaces discussed in this section resembles the correspondence between $\IR$-trees and ultrametric spaces discussed in \cite{Hug}, \cite{MPM}.
\end{remark}

\section{Tower morphisms}

In this section we shall discuss morphisms between towers.

\subsection{Introducing tower morphisms}
In this subsection we introduce several kinds of morphisms between towers $S,T$.

A map $\varphi:S\to T$ is defined to be
\begin{itemize}
\item {\em monotone} if for any $x,y\in S$ the inequality $x<y$ implies $\varphi(x)<\varphi(y)$;
\item {\em level-preserving} if there is an injective map $\varphi_{\Lev}:\Lev(S)\to\Lev(T)$ making the following diagram commutative:
$$
\begin{CD}
S@>{\varphi}>> T\\
@V{\lev_S}VV @ VV{\lev_T}V\\
\Lev(S)@>>{\varphi_{\Lev}}>\Lev(T).
\end{CD}
$$
\end{itemize}

For a monotone level-preserving map $\varphi:S\to T$ the induced map $\varphi_\Lev:\Lev(S)\to\Lev(T)$ is monotone and injective.

A monotone level-preserving map $\varphi:S\to T$ is called
\begin{itemize}
\item {\em a tower isomorphism} if it is bijective;
\item {\em a tower embedding} if it is injective;
\item {\em a tower immersion} if it is almost injective in the sense that for any points $x,x'\in S$ with $\varphi(x)=\varphi(x')$ we get $\lev_S(x\wedge x')\le\max\{\lev_S(x),\lev_S(x')\}+1$.
\end{itemize}

\begin{proposition}\label{p6} If $\varphi:S\to T$ is a tower embedding, then for any $x,x'\in S$ the inequality $x<x'$ is equivalent to $\varphi(x')<\varphi(x')$.
\end{proposition}

\begin{proof} If $x<x'$, then $\varphi(x)<\varphi(x')$ by the monotonicity of $\varphi$.

Now assume that $\varphi(x)<\varphi(x')$. The chain of the inequalities  $\varphi(x)\le\varphi(x')\le\varphi(x\wedge x')$ and the level-preserving property of $\varphi$ imply that $\lev(x)\le\lev(x')\le\lev(x\wedge x')$. Then there is a point $x''\in[x,x\wedge x']$ with $\lev(x'')=\lev(x')$. For this point $x''$ we get $\varphi(x)\le\varphi(x'')\le\varphi(x\wedge x')$.  Taking into account that   $\lev(\varphi(x''))=\lev(\varphi(x'))$ and the order interval $[\varphi(x),\varphi(x\wedge x')]\subset T$ is linearly ordered, we conclude that  $\varphi(x'')=\varphi(x')$ and $x''=x'$ by the injectivity of $\varphi$. Then $x\le x''=x'$ and $\varphi(x)\ne\varphi(x')$ implies  $x<x'$.\end{proof}

\subsection{Induced multi-maps between boundaries of towers}

Each monotone map $\varphi:S\to T$ between towers induces a multi-map $\partial\varphi:\partial S\Ra\partial T$ assigning to a branch $\beta\subset S$ the set $\partial\varphi(\beta)\subset\partial T$ of all branches of $T$ that contain the linearly ordered subset $\varphi(\beta)$ of $T$. It follows that $\partial\varphi(\beta)\ne\emptyset$ and hence $(\partial\varphi)^{-1}(\partial T)=\partial S$.

The following proposition describes some properties of the boundary multi-maps.

\begin{proposition}\label{p7} For a monotone map $\varphi:S\to T$ defined on a pruned tower $S$ the induced multi-map $\partial \varphi:\partial S\Ra\partial T$ is
\begin{enumerate}
\item single valued if $\forall \beta\in\partial S\;\forall\lambda\in\Lev(T)\;\exists x\in\beta$ with $\lev_T(\varphi(x))\le\lambda$;
\item micro-uniform if \ $\forall \lambda\in\Lev(T)\;\exists \nu\in\Lev(S)\;\forall x\in S$ \ $\lev_S(x)\le\nu\;\Ra\;\lev_T(\varphi(x))\le\lambda$;
\item macro-uniform if \ $\forall \nu\in\Lev(S)\;\exists \lambda\in\Lev(T)\;\forall x\in S$ \  $\lev_S(x)\le\nu\;\Ra\; \lev_T(\varphi(x))\le\lambda$.
\end{enumerate}
\end{proposition}

\begin{proof}  We recall that the boundaries $\partial S$ and $\partial T$ are endowed with ultrametrics $\rho_f$ and $\rho_g$ generated by some scaling functions $f:\Lev(S)\to[0,\infty)$ and $g:\Lev(T)\to[0,\infty)$.
\smallskip

1. Assuming that $\partial\varphi$ is not single-valued, we can find a branch $\beta\in\partial S$ and two distinct branches $b_1,b_2\in\partial T$ such that $b_1\cap b_2\supset\varphi(\beta)$. Since $b_1\ne b_2$, there is a level $\lambda\in\Lev(T)$ of $T$ such that the intersections $b_1\cap \lambda$ and $b_2\cap\lambda$ are not empty and distinct. For this level $\lambda$ no point $x\in\beta$ exists with $\lev_T(\varphi(x))\le\lambda$.
\smallskip

2. Assume that $\forall \lambda\in\Lev(T)\;\exists \nu\in\Lev(S)\;\forall x\in S$ \ $\lev_S(x)\le\nu\;\Ra\;\lev_T(\varphi(x))\le\lambda$. The micro-uniform property of  the boundary map $\partial \varphi:\partial S\Ra\partial T$  will follow as soon as for every $\e>0$ we find $\delta>0$ with $\w_{\partial\varphi}(\delta)\le\e$.

If the tower $T$ is ${\da}$-bounded, then the set $\Lev(T)$ has the smallest element $\lambda_0$. By our assumption, for the level $\lambda_0$ there is a level $\nu\in\Lev(S)$ such that $\lev_T(\varphi(x))\le\lambda_0$ for all $x\in S$ with $\lev_S(x)\le\nu$. Let $\delta=f(\nu)$. We claim that $\w_{\partial\varphi}(\delta)=0$.
This will follow as soon as we check that for each subset $A\subset\partial S$ with $\diam A\le\delta$ the image $\partial \varphi(A)$ is a singleton. Take any two branches $b_1,b_2\in\partial \varphi(A)$ and find two branches $a_1,a_2\in A$ with $b_i\in\partial\varphi(a_i)$ for $i\in\{1,2\}$. Since $\rho_f(a_1,a_2)\le\delta=f(\nu)$, there is a point $x\in a_1\cap a_2\cap\nu$. Then $\varphi(x)\in\lambda_0\cap b_1\cap b_2$ and the minimality of $\lambda_0$ implies that $b_1=b_2$.

Next, assume that the tower $T$ is $\da$-unbounded. In this case for every $\e>0$ we can find a level $\lambda\in\Lev(T)$ with $g(\lambda)\le \e$. By our hypothesis, for the level $\lambda$ there is a level $\nu\in\Lev(S)$ such that $\lev_T(\varphi(x))\le\lambda$ for each $x\in S$ with $\lev_S(x)\le\nu$. Let $\delta=f(\nu)$. We claim that $\w_{\partial \varphi}(\delta)\le\e$. This will follow as soon as we check that for each subset $A\subset\partial S$ with $\diam A\le\delta$ the image $\partial \varphi(A)$ has diameter $\le \e$. Take any two branches $b_1,b_2\in\partial \varphi(A)$ and find two branches $a_1,a_2\in A$ with $b_i\in\partial\varphi(a_i)$ for $i\in\{1,2\}$. Since $\rho_f(a_1,a_2)\le\delta$, there is a point $x\in a_1\cap a_2\cap\nu$. Since $\lev_T(\varphi(x))\le\lambda$ and $\varphi(x)\in b_1\cap b_2$, we get $\rho_g(b_1,b_2)\le g(\lambda)<\e$.
\smallskip

3. Assume that $\forall \nu\in\Lev(S)\;\exists \lambda\in\Lev(T)\;\forall x\in S$ \  $\lev_S(x)\le\nu\;\Ra\; \lev_T(\varphi(x))\le\lambda$. The macro-uniform property of the boundary map $\partial \varphi:\partial S\Ra\partial T$  will follow as soon as we check that for every $\delta<\infty$ the oscillation $\w_{\partial\varphi}(\delta)$ is finite.
Find a level $\nu\in\Lev(S)$ such that $f(\nu)\ge\delta$. By our hypothesis, for the level $\nu$ there is a level $\lambda\in\Lev(T)$ such that $\lev_T(\varphi(x))\le\lambda$ for each $x\in S$ with $\lev_S(x)\le\nu$. We claim that $\w_{\partial\varphi}(\delta)\le \e$ where $\e=g(\lambda)$. This will follow as soon as we check that for each subset $A\subset\partial S$ with $\diam A\le\delta$ the image $\partial \varphi(A)$ has diameter $\le \e$. Take any two branches $b_1,b_2\in\partial \varphi(A)$ and find two branches $a_1,a_2\in A$ with $b_i\in\partial\varphi(a_i)$ for $i\in\{1,2\}$. Since $\rho_f(a_1,a_2)\le\delta\le f(\nu)$, there is a point $x\in a_1\cap a_2\cap\nu$. Then $\lev_T(\varphi(x))\le\lambda$ and $\varphi(x)\in b_1\cap b_2$ implies that $\rho_g(b_1,b_2)\le g(\lambda)=\e$.
\end{proof}

Proposition~\ref{p7} implies

\begin{cor}\label{c6} For a level-preserving monotone  map $\varphi:S\to T$ defined on a pruned tower $S$ the induced multi-map $\partial \varphi:\partial S\Ra\partial T$ is
\begin{enumerate}
\item macro-uniform;
\item bi-uniform if $\varphi_\Lev(\Lev(S))$ is $\da$-cofinal in $\Lev(T)$.
\end{enumerate}
\end{cor}

Next, we establish some properties of the boundary multi-maps induced by tower immersions.

\begin{proposition}\label{p8} For a tower immersion $\varphi:S\to T$ defined on a pruned tower $S$ the induced multi-map $\partial \varphi:\partial S\Ra\partial T$ is
\begin{enumerate}
\item a macro-uniform embedding;
\item a bi-uniform embedding if the tower $S$ is $\da$-unbounded;
\item a macro-uniform equivalence if $\varphi(S)$ is $\da$-cofinal in $T$.
\item a bi-uniform equivalence if $S$ is $\da$-unbounded and $\varphi(S)$ is $\da$-cofinal in $T$.
\end{enumerate}
\end{proposition}

\begin{proof} Let $\varphi:S\to T$ be a tower immersion. It follows from the definition of $\partial\varphi$ that $(\partial\varphi)^{-1}(\partial T)=\partial S$.
The boundaries $\partial S$ and $\partial T$ are endowed with the ultrametrics $\rho_f$ and $\rho_g$ induced by some scaling functions $f:\Lev(S)\to(0,\infty)$ and $g:\Lev(T)\to(0,\infty)$.

\smallskip

1) Corollary~\ref{c6} implies that the boundary multi-map $\partial\varphi:\partial S\Ra\partial T$ is macro-uniform. It remains to check that the inverse multi-map $(\partial \varphi)^{-1}:\partial S\Ra\partial T$ is macro-uniform. This is clear if the tower $S$ is $\upa$-bounded (in which case $\partial S$ has finite diameter). So we assume that $S$ is ${\upa}$-unbounded. The tower immersion $\varphi$, being monotone and level-preserving, induces a monotone injective map $\varphi_\Lev:\Lev(S)\to\Lev(T)$. Now we see that $\varphi_\Lev(\Lev(S))$ is $\upa$-cofinal in $\Lev(T)$ and the tower $T$ is ${\upa}$-unbounded.

Given any finite $\delta$ we should find a finite $\e$ such that $\w_{(\partial\varphi)^{-1}}(\delta)\le\e$, which means that $\diam(\partial\varphi)^{-1}(A)\le\e$ for any subset $A\subset\partial T$ with $\diam A\le\delta$.
Since the tower $T$ is ${\upa}$-unbounded, there is a level $\lambda\in\Lev(T)$ such that $g(\lambda)\ge\delta$. The $\upa$-cofinality of the set $\varphi_\Lev(\Lev(S))$ in $\lev(T)$ allows us to  assume additionally that $\lambda=\varphi_\Lev(\nu)$ for some level $\nu\in\Lev(S)$.
We claim that the finite number $\e=f(\nu+1)$ has the desired property.
Take any two branches $b_1,b_2\in(\partial\varphi)^{-1}(A)$ and find two branches $a_1,a_2\in A$ with $b_i\in(\partial \varphi)^{-1}(a_i)$ for $i\in\{1,2\}$. The latter inclusion is equivalent to $a_i\in\partial\varphi(b_i)$. Since $\rho_g(a_1,a_2)\le\diam A\le\delta\le g(\lambda)$, there is a point $y\in\lambda\cap a_1\cap a_2$.

For every $i\in\{1,2\}$ let $x_i$ be the unique point of the intersection $b_i\cap\nu$. It follows that $\varphi(x_i)\subset\varphi(b_i)\cap\varphi(\nu)\subset a_i\cap\lambda=y$. Since $\varphi$ is a tower immersion, $\lev(x_1\wedge x_2)\le\max\{\lev(x_1),\lev(x_2)\}+1=\nu+1$. Then $x_1\wedge x_2\subset b_1\cap b_2$ and then $\rho_f(b_1,b_2)\le f(\lev(x_1\wedge x_2))\le f(\nu+1)=\e$.
\smallskip

2) Assume that the tower $S$ is ${\da}$-unbounded. Since the map  $\varphi_\Lev:\Lev(S)\to\Lev(T)$ is injective and monotone, the set $\varphi_\Lev(\Lev(S))$ is ${\da}$-cofinal in $\Lev(T)$ and the tower $T$ is $\da$-unbounded. By Corollary~\ref{c6}, the map $\partial \varphi:\partial S\Ra\partial T$ is bi-uniform and by the preceding item, the inverse multi-map $(\partial\varphi)^{-1}:\partial T\Ra\partial S$ is macro-uniform. It remains to check that this map micro-uniform. Since $S$ is $\da$-unbounded, for any $\e>0$ we can find a level $\nu\in\Lev(S)$ with $f(\nu+1)\le\e$. Since $T$ is $\da$-unbounded, we can find a level $\lambda\le\varphi_\Lev(\nu)$ in $T$. Repeating the argument from the preceding item we can show that the positive real number $\delta=g(\lambda)$ satisfies the inequality $\w_{(\partial\varphi)^{-1}}(\delta)\le\e$, witnessing that the multi-map $(\partial\varphi)^{-1}$ is micro-uniform.
\smallskip

3) The third statement will follow from the first one as soon as we check that $\partial\varphi(\partial S)=\partial T$ provided $\varphi(S)$ is $\da$-cofinal in $T$.

If $T$ is $\da$-bounded, then the ordered set $\lev(T)$ contains the smallest element $\lambda_0$. Then each branch $\beta\in\Lev(T)$ is equal to ${\upa}y$ where $\{y\}=\beta\cap\lambda_0$.
The cofinality of $\varphi(S)$ in $T$ implies that $\lambda_0\subset\varphi(S)$.
Take any point $x\in S$ with $\varphi(x)=y$ and observe that ${\upa}x$ is a branch in $\partial S$ whose image $\partial \varphi({\upa}x)={\upa}y=\beta$.

If $T$ is ${\da}$-unbounded, then so is the tower $S$. Let us show that the tower $T$ is pruned. Take any point $t\in T$ and use the cofinality of $\varphi(S)$ in $T$ in order to find a point $s\in S$ with $\varphi(s)\le t$. Since $S$ is pruned, there is a point $s'\in S$ with $s'<s$ and the monotonicity of $\varphi$ guarantees that $\varphi(s')<\varphi(s)\le t$, witnessing that $T$ is pruned.

Given any branch $\beta\in\partial T$ we are going to find a branch $\alpha\in\partial S$ with $\partial \varphi(\alpha)=\beta$.
Taking into account that the tower $T$ is pruned and $\da$-unbounded, we conclude that the branch $\beta$ meets all the levels of $T$. Fix a ${\da}$-cofinal subset $L\subset\Lev(S)$ such that $\lambda+1\notin L$ for every $\lambda\in L$.

For every $\lambda\in L$ pick a point $x_\lambda\in\lambda\cap\varphi^{-1}(\beta)$. Such a point $x_\lambda$ exists because $\beta$ meets the level $\varphi(\lambda)$ of $T$.
Let $x^+_\lambda$ be the unique point of the intersection ${\upa}x_\lambda\cap(\lambda+1)$.

We claim that the set $\{x^+_\lambda:\lambda\in L\}$ is linearly ordered. Indeed, take any two levels $\nu<\lambda$ and let $z_\lambda$ be the unique point of the intersection $\lambda\cap{\upa}(x^+_\nu)$. Taking into account that $$\varphi(z_\lambda)\ge\varphi(x_\nu^+)\in{\upa}\varphi(x_\nu^+)\subset\beta,$$ we see that $\varphi(z_\lambda)\in\beta\cap\varphi(\lambda)=\{\varphi(x_\lambda)\}$ and hence $\varphi(z_\lambda)=\varphi(x_\lambda)$. Since $\varphi$ is a tower immersion, $\lev(z_\lambda\wedge x_\lambda)\le \lambda+1$ and thus $x_\nu^+\le z_\lambda\wedge x_\lambda\le x_\lambda^+$.

The linearly ordered subset $\{x_\lambda^+:\lambda\in L\}$ can be enlarged to a branch $\alpha\in\partial S$ whose image $\partial\varphi(\alpha)$ coincides with the branch $\beta$.
\smallskip

4) If $\varphi(S)$ is cofinal in $T$ and the tower $S$ is $\da$-unbounded, then $\partial \varphi$ is a bi-uniform equivalence, being a surjective bi-uniform embedding according to the  statements (2) and (3) of Proposition~\ref{p8}.
\end{proof}

\subsection{Level subtowers}
It is clear that each $\upa$-directed subset $S$ of a tower $T$ is a tower with respect to the partial order inherited from $T$. In this case we say that $S$ is a {\em subtower} of $T$. A typical example of a subtower of $T$ is a {\em level subtower}  $$T^L=\{x\in T:\lev_T(x)\in L\},$$ where $L\subset\Lev(T)$ is an $\upa$-cofinal subset of the level set of the tower $T$.

Proposition~\ref{p8} implies

\begin{cor}\label{c7} Let $T$ be a pruned tower and $L$ be a ${\upa}$-cofinal subset of \ $\Lev(T)$. The multi-map $\partial\id:\partial T^L\Ra \partial T$ induced by the identity embedding $\id:T^L\to T$ is
\begin{enumerate}
\item a macro-uniform equivalence;
\item a bi-uniform equivalence if $L$ is ${\da}$-cofinal in $\Lev(T)$.
\end{enumerate}
\end{cor}

\subsection{Tower immersions induced by macro-uniform embeddings}

In Proposition~\ref{p8} we proved that for a tower immersion $\varphi:S\to T$ its boundary $\partial\varphi:\partial S\Ra\partial T$ is a macro-uniform embedding. It turns out that this statement can be partly reversed.

\begin{proposition}\label{p9} Let $S,T$ be pruned ${\upa}$-unbounded  towers. For any macro-uniform embedding $\Phi:\partial S\Ra\partial T$ there are $\da$-bounded ${\upa}$-cofinal  subsets $A\subset\Lev(S)$, $B\subset\Lev(T)$ and a tower immersion $\varphi:S^A\to T^B$ such that $$\partial\varphi=(\partial\id_T)^{-1}\circ\Phi\circ\partial \id_S$$ where $\partial\id_S:\partial S^A\Ra \partial S$ and $\partial \id_T:\partial T^B\Ra \partial T$ are boundary multi-maps, induced by the identity inclusions $\id_S:S^A\to S$ and $\id_T:T^B\to T$.
\end{proposition}

\begin{proof} Let $\Phi:\partial S\Ra\partial T$ be a macro-uniform embedding. We endow the boundaries $\partial S$ and $\partial T$ of the towers $S,T$ with the ultrametrics $\rho_f$, $\rho_g$ induced by some scaling functions $f:\Lev(S)\to[0,\infty)$ and $g:\Lev(T)\to[0,\infty)$. Let $\alpha_0$ be any level of the tower $S$.

By induction we can construct two increasing sequences $A=\{\alpha_n\}_{n\in\w}\subset\Lev(S)$ and $B=\{\beta_n\}_{n\in\w}\subset\Lev(T)$ such that
\begin{equation}\label{eq:p11}f(\beta_n)\ge\w_\Phi(g(\alpha_n)) \mbox{ \  and  \ }g(\alpha_{n+1})\ge \w_{\Phi^{-1}}(f(\beta_n))
\end{equation} for all $n\ge 0$.

Now we construct a tower immersion $\varphi:S^A\to T^B$. Given any point $s\in S^A$, find a level $\alpha_n$ containing $s$ and observe that the lower cone ${\da}s\subset S$ has  diameter $\diam{\da}s\le f(\alpha_n)$. Since $\diam \Phi({\da}s)\le \w_\Phi(f(\alpha_n))\le g(\beta_n)$, we conclude that $\Phi({\da}s)\subset{\da}\varphi(s)$ for a unique point $\varphi(s)\in \beta_n$.

It is clear that the so-defined map $\varphi:S^A\to T^B$ maps each level $\alpha_n$, $n\in\w$, into the level $\beta_n$, and hence is level-preserving. The uniqueness of the point $\varphi(s)$ with ${\da}\varphi(s)\supset\Phi({\da}s)$ implies that $\varphi$ is monotone.

To show that $\varphi$ is a tower immersion, take two points $s,s'\in\alpha_n$ and assume that $\varphi(s)=\varphi(s')=t$ for some point $t\in\beta_n\subset T$. Then $\Phi({\da}s)\cup\Phi({\da}s')\subset{\da}t$ and consequently, ${\da}s\cup{\da}s'\subset \Phi^{-1}({\da}t)$. It follows from the choice of $\alpha_{n+1}$ that
$$\diam({\da}s\cup{\da}s')\le\diam \Phi^{-1}({\da}t)\le f(\alpha_{n+1})$$ which implies that $s,s'\in{\da}s''$ for some point $s''\in\alpha_{n+1}$. Consequently, $\lev_{S^A}(s\wedge s')\le \alpha_{n+1}$ and the level $\alpha_{n+1}$ is the successor level of $\alpha_n=\lev(s)=\lev(s')$ in the tower $S^A$, witnessing that the map $\varphi:S^A\to T^B$ is a tower immersion.

The definition of $\varphi$ easily implies that  $\partial\varphi=(\partial\id_{T^B})^{-1}\circ\Phi\circ\partial \id_{S^A}$.
\end{proof}

By analogy we can prove

\begin{proposition}\label{p10} Let $S,T$ be pruned ${\dupa}$-unbounded  towers. For any bi-uniform embedding $\Phi:\partial S\to\partial T$ there are ${\dupa}$-cofinal  subsets $A\subset\Lev(S)$, $B\subset\Lev(T)$ and a tower immersion $\varphi:S^A\to T^B$ such that $$\partial\varphi=(\partial\id_T)^{-1}\circ\Phi\circ\partial \id_S$$ where $\partial\id_S:\partial S^A\to \partial S$ and $\partial \id_T:\partial T^B\to \partial T$ are bi-uniform equivalences, induced by the identity inclusions $\id_S:S^A\to S$ and $\id_T:T^B\to T$.
\end{proposition}

\subsection{Constructing tower embeddings and isomorphisms} In this subsection we describe a method of constructing tower embedding and isomorphisms.

\begin{proposition}\label{p11} Let $S,T$ be pruned towers and
$f:\Lev(S)\to \Lev(T)$ be a  monotone (and surjective) map. If $\Deg_\lambda^{\lambda+1}(S)\le\deg_{f(\lambda)}^{f(\lambda+1)}(T)$ (and $\deg_\lambda^{\lambda+1}(S)\ge\Deg_{f(\lambda)}^{f(\lambda+1)}(T)$) for each non-maximal level $\lambda\in\Lev(S)$, then there is a tower embedding (a tower isomorphism) $\varphi:S\to T$ such that $\varphi_\lev=f$.
\end{proposition}

\begin{proof}  A map $\varphi:A\to T$ defined on a subset $A\subset S$ will be called an {\em $f$-map} if $\lev_T(\varphi(a))=f(\lev_S(a))$ for every $a\in A$. If, in addition, $\varphi$ is a tower embedding (isomorphism), then $\varphi$ will be called $f$-embedding ($f$-isomorphism). The proof of Proposition~\ref{p11} is based on the following lemma.

\begin{lemma}\label{l3} For any two points $u\in S$ and $v\in T$ with $f(\lev_S(u))=\lev_T(v)$ there is an $f$-embedding ($f$-isomorphism) $\varphi:{\downarrow}u\to{\downarrow}v$. Moreover, if for some $u_0\in\suc(u)$ and $v_0\in\suc_{f(\lev\, u_0)}(v)$ we are given with a tower $f$-embedding ($f$-isomorphism) $\varphi_0:{\downarrow}u_0\to {\downarrow}v_0$, then the map $\varphi$ can be chosen so that $\varphi|{\downarrow}u_0=\varphi_0$.
\end{lemma}

\begin{proof} For every level $\lambda\le\lev_S(u)$ of $S$ consider the  subtower      $S_\lambda(u)=\{s\in {\da}u:\lev(s)\ge \lambda\}$ having finitely many levels.
By induction we are going to construct an $f$-embedding $\varphi_\lambda:S_\lambda(u)\to T$ so that $\varphi_{\lambda-1}$ extends $\varphi_\lambda$.

If $\lambda=\lev_S(u)$, then $S_\lambda(u)=\{u\}$ and we can put $\varphi_\lambda(u)=v$.
Assume that for some level $\lambda<\lev_S(u)$ of $S$ an $f$-embedding $\varphi_{\lambda+1}:S_{\lambda+1}(u)\to T$ has been constructed.
Observe that
$$S_\lambda(u)=S_{\lambda+1}(u)\cup\bigcup\{\suc(x):x\in (\lambda+1)\cap{\da}u\}.$$
By our assumption, for every $x\in (\lambda+1)\cap{\da}u$, we get $$\deg(x)\le\Deg_{\lambda}^{\lambda+1}(S)\le\deg_{f(\lambda)}^{f(\lambda+1)}(T)\le
\deg_{f(\lambda)}^{f(\lambda+1)}(f(x)).$$
Consequently, we can find an injective map $\psi_x:\suc_\lambda(x)\to\suc_{f(\lambda)}(f(x))$. Moreover, if $\deg_\lambda(x)=\deg_{f(\lambda)}(f(x))$, then we can take the map $\psi_x$ to be bijective. If for some $u_0\in\suc(u)$ and $v_0\in\suc_{f(\lev\,u_0)}(v)$ we are given with a tower $f$-embedding ($f$-isomorphism) $\varphi_0:{\downarrow}u_0\to {\downarrow}v_0$, then we can assume that $\psi_x=\varphi_0|\suc(x)$ if $x\le u_0$.

Now define the $f$-embedding $\varphi_\lambda:S_\lambda\to T$ letting $\varphi_\lambda|S_{\lambda+1}=\varphi_{\lambda+1}$ and $\varphi_\lambda|\suc_\lambda(x)=\psi_x$ for $x\in (\lambda+1)\cap{\da}u$.
This completes the inductive step.

One can readily check that the $f$-embedding $\varphi:{\da}u\to{\da}v$ defined by  $\varphi|S_\lambda(u)=\varphi_\lambda$ for levels $\lambda\le\lev_S(u)$ of $S$ has the required properties.
\end{proof}

Now let us return to the proof of Proposition~\ref{p11}. Fix any point $\theta_S\in S$ and for every level $\lambda\ge\lev_S(\theta_S)$ of the tower $S$ denote by $u_\lambda$  the unique point of the intersection ${\upa}\theta_S\cap\lambda$. Choose any point $\theta_T$ at the level $f(\lev_S(\theta_S))\subset T$ and for every level $\lambda\ge\lev_T(\theta_T)$ denote by $v_\lambda$ the unique point of the intersection $\lambda\cap{\upa}\theta_T$.

For the initial level $\lambda=\lev_S(\theta_S)$ we can apply the first part of Lemma~\ref{l3} in order to find an $f$-embedding (an $f$-isomorphism) $\varphi_\lambda:{\da}u_\lambda\to{\da}v_{f(\lambda)}$. Applying inductively the second part of Lemma~\ref{l3}, for every level $\lambda>\lev_S(\theta_S)$ of $S$ we can find an $f$-embedding ($f$-isomorphism) $\varphi_\lambda:{\da}u_\lambda\to{\da}v_{f(\lambda)}$ such that $\varphi_\lambda|{\da}u_{\lambda-1}=\varphi_{\lambda-1}$.

After completing the inductive construction, we define an $f$-embedding ($f$-isomorphism) $\varphi:S\to T$ letting $\varphi|{\da}u_\lambda=\varphi_\lambda$ for $\lambda\ge\lev_S(\theta_S)$. The $f$-embedding $\varphi$ is well-defined because $S$ is upward directed and hence $S=\bigcup_{\lambda\ge\lev_S(\theta_S)}{\da}u_\lambda$.
\end{proof}


Applying Proposition~\ref{p11} to homogeneous towers we get

\begin{cor}\label{c8} Two homogeneous towers $S,T$ are isomorphic if and only if there is an order isomorphism $f:\Lev(S)\to\Lev(T)$ such that  $\deg_\lambda^{\lambda+1}(S)=\deg_{f(\lambda)}^{f(\lambda+1)}(T)$ for each  non-maximal level $\lambda\in\Lev(S)$.
\end{cor}

\section{The Key Lemma}

The principal result of this section is Lemma~\ref{MainLemma}, which is the most difficult result of this paper. This lemma allows us to construct immersions between \mbox{${\da}$-bounded} towers and will be used in the proof of Theorems~\ref{MU-char} and \ref{bU-char} in Sections~\ref{pf:MU} and \ref{pf:bU}.

It follows from Corollary~\ref{c7} that the boundary $\partial T$ of each tower $T$ is macro-uniformly equivalent to the boundary $\partial T^L$ of the level subtower $T^L$ for any $\upa$-cofinal subset $L\subset\Lev(T)$. The subset $L$ can be chosen to be $\da$-bounded in $\Lev(T)$, which implies that the level subtower $T^L$ is $\da$-bounded. Therefore, for studying the macro-uniform structure of ultrametric spaces it suffices to restrict ourselves by $\da$-bounded  $\upa$-unbounded towers $T$.

In this case the level set $\Lev(T)$ of $T$ has the smallest element and can be canonically labeled by finite ordinals. For $k\in\w$ by $\Lev_k(T)$ we shall denote the $k$-th level of $T$. The identification of $\Lev(T)$ with $\w$ defines the canonical scaling function $\id:\Lev(T)\to\w\subset [0,\infty)$  that induces the canonical ultrametric $\rho_\id$ on the boundary $\partial T$ of $T$.
Observe that $\partial T$ can be identified with the smallest level $\Lev_0(T)$ of $T$.

\begin{lemma}\label{MainLemma} For a $\da$-bounded tower $T$ and a $\da$-bounded homogeneous tower $H$ there is a surjective tower immersion $\varphi:T\to H$ if the following two inequalities hold for every $k\in\IN$
\begin{enumerate}
\item $\deg^k_0(T)\ge 4^{k+5}\cdot\deg_0^{k-1}(H)$ and
\item $\deg_0^{k}(H)\ge 4^{k}\cdot\Deg_0^{k}(T)$.
\end{enumerate}
\end{lemma}

\begin{proof} First we introduce some notation.

A subset $A$ of the tower $T$ will be called a {\em trapezium} if $A={\da}P$ for some non-empty subset $P\subset\suc(v)$ of parents of some point $v\in T$, called the {\em vertex} of the trapezium $A$ and denoted by $\vx(A)$. It is easy to see that $\{\vx(A)\}\cup{\da}P$ is a subtower of $T$. The set $P$ generating the trapezium $A={\da}P$ is called the {\em plateau} of the trapezium. For the plateau $P$ let $\deg_0(P)=|{\da}P\cap\Lev_0(T)|$ be the cardinality of the ``base'' ${\da}P\cap \Lev_0(T)$ of the trapezium ${\da}P$.

A map $\varphi:{\da}P\to H$ from a trapezium ${\da}P\subset S$ to the tower $H$ will be called an {\em admissible immersion} if
\begin{itemize}
\item $\varphi=\phi|{\da}P$ for some tower immersion $\phi:\{\vx({\da}P)\}\cup{\da}P\to H$,
\item $\varphi(P)=\{t\}$ for some $t\in T$,
\item $\varphi({\da}P)={\da}t$.
\end{itemize}

Let $\e_k=\frac{1}{4^k}$, $k\in\IN$, and observe that  $$\prod_{k=1}^\infty\frac{1+\e_k}{1-\e_k}<2.$$

Lemma~\ref{MainLemma} will be derived from the following

\begin{claim}\label{cl1} For any $k\in\IN$, a trapezium ${\da}A_k\subset T$, and a vertex $w\in H$ at the height $k=\lev(A_k)=\lev(w)$ there is an admissible immersion $\varphi:{\downarrow}A_k\to {\downarrow}w$ provided
$$4\le 8\cdot\prod_{i=k+1}^\infty\frac{1-\e_i}{1+\e_i}\le\frac{\deg_0(A_k)}{\deg^k_0(H)}
\le 16 \prod_{i=k+1}^\infty\frac{1+\e_i}{1-\e_i}\le 32.$$
\end{claim}

\begin{proof} This claim will be proved by induction on $k$. If $k=0$, then
${\da}A_k=A_k$ and the constant map $\varphi:A_k\to\{w\}\subset H$ is the required immersion.

Assume that the claim has been proved for some $k-1\in\w$. Fix a trapezium ${\da}A_k\subset S$ and a point $w\in T$ with $\lev_S(A_k)=\lev_T(w)=k$ so that the upper and lower bounds from Claim~\ref{cl1} hold.

Since $\deg_0(A_k)=\sum_{a\in A_k}\deg_0(a)$, for every point $a\in A_k$ we can choose an integer number $d_a$ such that
\[\label{eqnew1}
\Big|d_a-\deg^k_{k-1}(H)\frac{\deg_0(a)}{\deg_0(A_k)}\Big|\le 1\]and
 $\sum_{a\in A_k}d_a=\deg^k_{k-1}(H)=\deg(w)$.

\begin{claim}\label{cl2} For every $a\in A_k$ $$\frac{\deg^k_{k-1}(H)}{\deg_0(A_k)}(1-\e_k)\le \frac{d_a}{\deg_0(a)}\le \frac{\deg^k_{k-1}(H)}{\deg_0(A_k)}(1+\e_k).$$
\end{claim}

\begin{proof} It follows from the choice of $d_a$ that
$$\frac{d_a}{\deg_0(a)}\le \frac{\deg^k_{k-1}(H)}{\deg_0(A_k)}+\frac1{\deg_0(a)}=
\frac{\deg^k_{k-1}(H)}{\deg_0(A_k)}\cdot\Big(1+\frac{\deg_0(A_k)}{\deg^k_{k-1}(H)\cdot\deg_0(a)}\Big).$$The upper bound in Claim~\ref{cl1} implies
$$
\frac{\deg_0(A_k)}{\deg^k_{k-1}(H)\cdot\deg_0(a)}\le \frac{32\cdot\deg^k_0(H)}{\deg^k_{k-1}(H)\cdot\deg_0(a)}\le \frac{32\cdot\deg_0^{k-1}(H)}{\deg_0^k(T)}\le \frac{1}{4^{k}}=\e_k.$$
The last inequality follows from the condition (1) of Lemma~\ref{MainLemma}.

This proves the upper bound of Claim~\ref{cl2}. By analogy we can prove the lower bound.
\end{proof}

Claim~\ref{cl2}, the upper bound of Claim~\ref{cl1} and the condition (1) of Lemma~\ref{MainLemma} imply
$$
d_a\ge \deg_0(a)\frac{\deg^k_{k-1}(H)}{\deg_0(A_k)}(1-\e_k)\ge \frac{\deg^k_0(T)\cdot\deg^k_{k-1}(H)}{32\deg^k_0(H)}\frac12\ge \frac{4^{k+5}\cdot\deg^{k-1}_0(H)}{64\cdot\deg^{k-1}_0(H)}\ge
4^{k-1}>0.
$$

For every $a\in A_k$ write the set $\suc(a)$ of parents of $a$ in the tower $T$ as the disjoint union $\suc(a)=\cup \A_a$ of a family $\A_a$ containing $d_a$ sets such that for every $A_{k-1}\in\A_a$
we get $$\Big|\deg_0(A_{k-1})-\frac{\deg_0(a)}{d_a}\Big|\le \Deg^{k-1}_0(T).$$

\begin{claim}\label{cl3} For each set $A_{k-1}\in\A_a$ the upper and lower bounds of Claim~\ref{cl1} are satisfied for $k-1$.
\end{claim}

\begin{proof} If $k=1$, then
$$\Big|\deg_0(A_{0})-\frac{\deg_0(a)}{d_a}\Big|\le \Deg^{0}_0(T)=1$$and by Claim~\ref{cl2} and the inductive assumption:
 $$
\begin{aligned}
\deg_0(A_{0})&\le \frac{\deg_0(a)}{d_a}+1\le \frac{\deg_0(A_1)}{\deg^1_{0}(H)(1-\e_1)}+1\le\\
&\le\frac{\deg_0(A_1)}{\deg_0^1(H)(1-\e_k)}
\left(1+\frac{\deg^{1}_0(H)}{\deg_0(A_1)}\right)\le
\frac{\deg_0(A_1)}{\deg_0^1(H)(1-\e_1)}
\left(1+\frac14\right)\le 16\prod_{i=1}^\infty\frac{1+\e_i}{1-\e_i}.
\end{aligned}
$$
By analogy, we can prove the lower bound
$$\deg_0(A_0)\ge \frac{\deg_0(A_1)}{\deg_0^1(H)}\cdot\frac{1-\e_1}{1+\e_1}\ge8\prod_{i=1}^\infty\frac{1-\e_i}{1+\e_i}.$$

Next, assume that $k>1$. Then by Claim~\ref{cl2}:
$$\begin{aligned}
\frac{\deg_0(A_{k-1})}{\deg^{k-1}_0(H)}&\le \frac1{\deg^{k-1}_0(H)}\cdot\frac{\deg_0(a)}{ d_a}+\frac{\Deg^{k-1}_0(T)}{\deg^{k-1}_0(H)}\le \frac{1}{\deg^{k-1}_0(H)}\cdot\frac{\deg_0(A_k)}{\deg^k_{k-1}(H)(1-\e_k)}+\frac{\Deg^{k-1}_0(T)}{\deg^{k-1}_0(H)}\le \\
&\le
\frac{\deg_0(A_k)}{\deg_0^k(H)(1-\e_k)}\left(1+\frac{\Deg^{k-1}_0(T)\deg^k_0(H)}{\deg^{k-1}_0(H)\deg_0(A_k)}\right)
\end{aligned}
$$and
$$\frac{\Deg^{k-1}_0(T)\deg^k_0(H)}{\deg^{k-1}_0(H)\deg_0(A_k)}=
\frac{\Deg^{k-1}_0(T)\deg^k_{k-1}(H)}{\deg_0(A_k)}\le \frac{\Deg^{k-1}_0(T)\deg^k_{k-1}(H)}{4\deg^k_0(H)}=\frac{\Deg^{k-1}_0(T)}{4\deg_0^{k-1}(H)}\le
\frac1{4\cdot4^k}\le\e_k$$
by the lower bound from Claim~\ref{cl1} and the condition (2) of Lemma~\ref{MainLemma}.
Then
$$\frac{\deg_0(A_{k-1})}{\deg^{k-1}_0(H)}\le \frac{\deg_0(A_k)}{\deg^k_0(H)}\cdot \frac{1+\e_k}{1-\e_k}\le 16\cdot\Big(\prod_{i=k+1}^\infty\frac{1+\e_i}{1-\e_i}\;\Big)\cdot\frac{1+\e_k}{1-\e_k}=16\cdot\prod_{i=k}^\infty\frac{1+\e_i}{1-\e_i}.
$$

By analogy, we can prove that
$$
\frac{\deg_0(A_{k-1})}{\deg^{k-1}_0(H)}\ge \frac{\deg_0(A_k)}{\deg^k(H)}\cdot\frac{1-\e_k}{1+\e_k}\ge 8\cdot\prod_{i=k}^\infty\frac{1+\e_i}{1-\e_i}.
$$
\end{proof}

The family $\A=\bigcup_{a\in A_k}\A_a$ has cardinality $|\A|=\sum_{a\in A_k}|\A_a|=\sum_{a\in A_k}d_a=\deg(w)$ and hence we can find a bijective map $f:\A\to\suc(w)$. By the inductive assumption and Claim~\ref{cl3}, for each set $A'\in\A$ we can find an admissible immersion $\varphi_{A'}:{\downarrow} A'\to{\downarrow} f(A')$. Now define the admissible immersion $\varphi:{\downarrow}P\to{\downarrow}w$ letting
$$\varphi(x)=
\begin{cases}
\varphi_{A'}(x)&\mbox{ if $x\in{\downarrow}A'$ for some $A'\in\A$};\\
w&\mbox{ if $x\in A_k$}.
\end{cases}
$$
This completes the proof of Claim~\ref{cl1}.
\end{proof}

Now we are able to complete the proof of Lemma~\ref{MainLemma}. Let $(a_k)_{k\in\w}$ and $(b_k)_{k\in\w}$ be two branches of the towers $T$ and $H$, respectively.
For every $k\in\w$ choose a subset $A_k\subset\suc(a_{k+1})$ such that $a_k\in A_k$ and
$$11\le\frac{\deg_0(A_k)}{\deg^k_0(H)}\le 13.$$Such a choice of $A_k$ is always possible because
$\deg_0(a_{k+1})\ge\deg^{k+1}_0(T)\ge 4^{k+6}\deg^k_0(H)$ by the condition (1) of Lemma~\ref{MainLemma} and
$\displaystyle\frac{\Deg^k_0(T)}{\deg^k_0(H)}\le\frac{1}{4^k}\le1$ by the condition (2) of Lemma~\ref{MainLemma}.

 By induction for every $k\in\w$ we shall construct a tower immersion $\varphi_k:{\da}A_k\to{\da}b_k$ such that $\varphi_{k-1}=\varphi_k|{\da}A_{k-1}$.

For $k=0$ the constant map $\varphi_0:A_0\to\{b_0\}$ is the desired immersion.
Assume that for some $k\in\w$ an immersion $\varphi_k:{\da}A_k\to{\da}b_k$ has been constructed. Consider the trapezium ${\da}A$ with the plateau $$A=(\Lev_k(T)\cap{\da}A_{k+1})\setminus A_k$$ in the tower $T$. Also consider the trapezium ${\da}B$ with plateau $B=\suc(b_{k+1})\setminus\{b_k\}$ in the homogeneous tower $H$.
It is clear that $\deg_0(A)=\deg_0(A_{k+1})-\deg_0(A_k)$ and
$|B|=\deg^{k+1}_k(H)-1$.
Observe that
$$\deg^{k+1}_k(H)=\frac{\deg^{k+1}_0(H)}{\deg^k_0(H)}\ge 4^{k+1}\frac{\Deg^{k+1}(T)}{\deg^k_0(H)}\ge 4^{k+1}\frac{4^{k+6}\deg^k_0(H)}{\deg^k_0(H)}=4^{2k+7}\ge 4^7.$$

Write $A$ as the disjoint union $A=\bigcup_{b\in B}A_b$ of subsets $A_b\subset A$ such that for every $b\in B$
$$\big|\deg_0(A_b)-\frac{\deg_0(A)}{|B|}\big|\le \Deg^k_0(T).$$It follows from the condition (2) of Lemma~\ref{MainLemma} that
$$
\begin{aligned}
\frac{\deg_0(A_b)}{\deg^k_0(H)}&\le \frac1{\deg^k_0(H)}\Big(\frac{\deg_0(A)}{\deg^{k+1}_k(H)-1}+\Deg^k_0(T)\Big)\le\\
&
\le\frac1{\deg^k_0(H)}\cdot \frac{\deg^{k+1}_k(H)}{\deg^{k+1}_k(H)-1}\cdot
\frac{\deg_0(A_{k+1})}{\deg^{k+1}_k(H)}+\frac{\Deg^k_0(T)}{\deg^k_0(H)}\le\\
&\le\frac{4^7}{4^7-1}\cdot
\frac{\deg_0(A_{k+1})}{\deg^{k+1}_0(H)}+\frac{1}{4^k}\le\frac{14}{13}\cdot
13+1<16.
\end{aligned}
$$
On the other hand,
$$
\begin{aligned}
\frac{\deg_0(A_b)}{\deg^k_0(H)}&\ge \frac1{\deg^k_0(H)}\Big(\frac{\deg_0(A)}{\deg^{k+1}_k(H)-1}-\Deg^k_0(T)\Big)\ge\\
&
\ge\frac1{\deg^k_0(H)}\cdot
\frac{\deg_0(A_{k+1})-\deg_0(A_k)}{\deg^{k+1}_k(H)}-\frac{\Deg^k_0(T)}{\deg^k_0(H)}\ge\\
&\ge\frac{11\deg^{k+1}_0(H)-13\deg^k_0(H)}{\deg^{k+1}_0(H)}-\frac{1}{4^k}\ge11-\frac{13}{4^7}-1\ge 8.
\end{aligned}
$$

Two above two inequalities imply that the trapezium ${\da}A_b$ satisfies the upper and lower bounds of Claim~\ref{cl1}, which yields an admissible immersion $\varphi_b:{\da}A_b\to{\da}b$. The immersions $\varphi_b$ compose the immersion $\varphi_{k+1}:{\da}A_{k+1}\to{\da}b_{k+1}$ defined by the formula:
$$\varphi_{k+1}(x)=\begin{cases}
\varphi_k(x)&\mbox{if $x\in{\da}A_k$},\\
\varphi_b(x)&\mbox{if $x\in {\da}A_b$ for some $b\in B$}.
\end{cases}
$$
Since $\varphi_k=\varphi_{k+1}|{\da}A_k$ for all $k\in\w$ we can define an immersion $\varphi:T\to H$ letting $\varphi|{\da}a_k=\varphi_k$ for $k\in\w$.
\end{proof}

\section{Proof of Theorem~\ref{MU-char} (Macro-Uniform Characterization of the Cantor bi-cube)}\label{pf:MU}
\label{pt3}

The ``only if'' part of Theorem~\ref{MU-char} follows from Lemmas~\ref{l1} and \ref{l2}.
To prove the ``if'' part, assume that a metric space $X$ has macro-uniform dimension zero and for some $\delta>0$ we get $\Theta^\e_\delta(X)<\infty$ for all $\e\ge\delta$ and $\lim\limits_{\e\to\infty}\theta^\e_\delta(X)=\infty$.

Let $\lambda_0=\delta$ and $m_0=0$. By induction we can construct increasing sequences $(\lambda_k)_{k=0}^\infty\subset(0,+\infty)$ and $(m_k)_{k=0}^\infty\subset\w$ such that
 $\theta_{\delta}^{\lambda_k}(X)\ge 4^{k+5}\cdot 2^{m_{k-1}}$ and
 $2^{m_k}\ge 4^{k}\cdot\Theta_{\delta}^{\lambda_k}(X)$
for all $k\in\IN$.

Let $L=\{\lambda_n\}_{n\in\IN}\subset(0,\infty)$ and consider the canonical $L$-tower $T_X^L=\{(C_{\lambda}(x),\lambda):x\in X,\;\lambda\in L\}$ of the metric space $X$. Its level set $\Lev(T_X^L)$ can be identified with the set $L$. By Corollary~\ref{c5}, the canonical map
$$C_L:X\to\partial T_X^L,\;\;C_L:x\mapsto C_L(x)=\{(C_{\lambda}(x),\lambda):\lambda\in L\},$$ is a macro-uniform equivalence.

Next, consider an ${\da}$-unbounded binary tower $T_2$. Its level-set $\Lev(T_2)$ can be identified with $\IZ$ and we can consider the level subtower $T_2^M\subset T_2$ where $M=\{m_k\}_{k\in\w}\subset\IZ$. By Corollary~\ref{c7}, the boundary multi-map $\partial\id_{T_2^M}:\partial T_2^M\Ra \partial T_2=2^{<\IZ}$ induced by the identity embedding $\id_{T_2^M}:T_2^M\to T_2$ is a macro-uniform equivalence.

 Observe that $H=T_2^M$ is a homogeneous tower and
$$
\deg^k_0(T_2^M)=2^{m_{k}},\;\;
\deg^k_0(T^L_X)=\theta^{\lambda_k}_\delta(X),\;\;
\Deg^k_0(T^L_X)=\Theta^{\lambda_k}_\delta(X)
$$
which allows us to apply Lemma~\ref{MainLemma} to constructing a surjective tower immersion $\varphi:T^L_X\to T_{2}^M$. By Proposition~\ref{p8}(3), $\varphi$ induces a macro-uniform equivalence $\partial\varphi:\partial T^L_X\Ra\partial T_2^M$. Finally we obtain a macro-uniform equivalence  between $X$ and the Cantor bi-cube $2^{<\IZ}$ as the composition  of the macro-uniform equivalences
$$X\sim \partial T^L_X\sim \partial T_2^M\sim \partial T_2=2^{<\IZ}.$$

\section{Proof of Theorem~\ref{bU-char}  (Bi-Uniform Characterization of the Cantor bi-cube)}\label{pf:bU}

The ``only if'' part of Theorem~\ref{bU-char} easily follows from Lemmas~\ref{l1} and \ref{l2}.
To prove the ``if'' part, assume that  $X$ is a complete metric space of bi-uniform dimension zero such that for every $0<\delta\le\e<\infty$ the number $\Theta^\e_\delta(X)$ is finite and for every $0<\e<\infty$
$$\lim_{\delta\to+0}\theta^\e_\delta=\infty=\lim_{\delta\to+\infty}\theta^\delta_\w(X).$$

Let $\lambda_0=1$ and $m_0=0$. By induction construct increasing sequences $(\lambda_k)_{k=0}^\infty\subset[1,\infty)$ and $(m_k)_{k=0}^\infty\subset\w$ such that
for every $k\in\w$ the following conditions hold:
\begin{itemize}
\item[(i)] $\theta_{\lambda_0}^{\lambda_k}(X)\ge 4^{k+5}\cdot 2^{m_{k-1}}$ and
\item[(ii)] $2^{m_k}\ge 4^{k}\,\Theta_{\lambda_0}^{\lambda_k}(X)$.
\end{itemize}

By reverse induction, construct sequences $(\lambda_{k})_{k=-\infty}^1\subset(0,1)$ and $(m_k)_{k=-\infty}^1\subset\IZ$ such that
\begin{itemize}
\item[(iii)] $\lambda_{k-1}<\lambda_{k}$ and $m_{k-1}<m_{k}$ for each $k\le 0$;
\item[(iv)] $\lim\limits_{k\to-\infty}\lambda_k=0$, $\lim\limits_{k\to-\infty}m_k=-\infty$, and
\item[(v)] $\Theta^{\lambda_{k+1}}_{\lambda_{k}}(X)\le 2^{m_{k}-m_{k-1}}\le \theta^{\lambda_{k}}_{\lambda_{k-1}}(X)$.
\end{itemize}

For the subset $L=\{\lambda_n:n\in\IZ\}\subset(0,+\infty)$, consider the canonical $L$-tower $T_X^L=\{(C_{\lambda}(x),\lambda):x\in X,\;\lambda\in L\}$ of the metric space $X$. By Corollary~\ref{c5}(3), the canonical map
$$C_L:X\to\partial T_X^L,\;\;C_L:x\mapsto C_L(x)=\{(C_{\lambda}(x),\lambda):\lambda\in L\},$$ is a bi-uniform equivalence.

Next, consider an ${\da}$-unbounded binary tower $T_2$.
Its level-set $\Lev(T_2)$ can be identified with $\IZ$ and we can consider its level subtower $T_2^M\subset T_2$ where $M=\{m_k\}_{k\in\IZ}\subset\IZ$. By Corollary~\ref{c7}, the boundary map $\partial\id_{T_2^M}:\partial T_2^M\Ra \partial T_2=2^{<\IZ}$ induced by the identity embedding $\id_{T_2^M}:T_2^M\to T_2$ is a bi-uniform equivalence.

For every $n\in\IZ$ let $L_n=\{\lambda_k:k\ge n\}$ and $M_n=\{m_k:k\ge n\}$. Repeating the argument of the proof of Theorem~\ref{MU-char} and applying
Lemma~\ref{MainLemma}, we can find a surjective tower immersion $\varphi_0:T^{L_0}_X\to T_{2}^{M_0}$.
Now our aim is to extend the immersion $\varphi_0$ to a tower immersion $\varphi: T^L_X\to T_2^M$.

By induction we shall define surjective tower immersions $\varphi_k:T_X^{L_k}\to T_2^{M_k}$, $k\le0$, such that $\varphi_{k-1}|T_X^{L_{k}}=\varphi_k$ for all $k\le 0$.

Assuming that for some $k\le 0$ a surjective tower immersion $\varphi_k:T_X^{L_k}\to T_2^{M_k}$ has been defined, we construct a tower immersion $\varphi_{k-1}:T_X^{L_{k-1}}\to T_2^{M_{k-1}}$ as follows. Since $\varphi_k$ is a tower immersion, for every point $y\in T_2^{M_k}$ at the lowest level $m_k$ of the tower $T_2^{M_k}$ the preimage $\varphi^{-1}_k(y)$ lies in the set $\suc_{\lambda_{k}}(s)=\lambda_{k}\cap {\da}s$ of parents of some point $s\in \lambda_{k+1}$. Consequently, $|\varphi^{-1}_k(y)|\le\deg_{\lambda_{k}}(s)\le\Deg_{\lambda_{k}}^{\lambda_{k+1}}(T_X)=\Theta_{\lambda_k}^{\lambda_{k+1}}(X)$.
By the choice of $\lambda_{k-1}$, we get
$$|\varphi^{-1}(y)|\le \Theta_{\lambda_k}^{\lambda_{k+1}}(X)\le 2^{m_k-m_{k-1}}=\deg_{m_{k-1}}^{m_k}(T_2)\le\deg_{m_{k-1}}(y)=|\suc_{m_{k-1}}(y)|$$and
consequently we can find a surjective map $\psi_y:\suc_{m_{k-1}}(y)\to \varphi_n^{-1}(y)$.
By the choice of $\lambda_{k-1}$, for every $x\in\varphi_k^{-1}(y)\subset \lambda_k$ we get
$$
\begin{aligned}
|\suc_{\lambda_{k-1}}(x)|&=\deg_{\lambda_{k-1}}(x)\ge\deg_{\lambda_{k-1}}^{\lambda_k}(T_X)=\theta_{\lambda_{k-1}}^{\lambda_k}(X)\ge 2^{m_k-m_{k-1}}=
\Deg_{m_{k-1}}^{m_k}(T_2)=\\
&= \deg_{m_{k-1}}(y)=|\suc_{m_{k-1}}(y)|\ge|\psi_y^{-1}(x)|
\end{aligned}$$and consequently, we can find a surjective map
$\varphi_x:\suc_{\lambda_{k-1}}(x)\to \psi_y^{-1}(x)$. Now define the tower immersion   $\varphi_{n-1}:T_X^{L_{k-1}}\to T_2^{M_{k-1}}$ by the formula
$$\varphi_{k-1}=\varphi_k\cup\bigcup_{y\in m_k}\bigcup_{x\in\varphi_k^{-1}(y)}\varphi_x.$$

After completing the inductive construction, we can see that $$\varphi=\bigcup_{n\le 0}\varphi_n:T_X^L\to T_2^{M}$$ is a tower immersion. By Proposition~\ref{p8}(4), the tower immersion $\varphi$ induces a bi-uniform equivalence $\partial\varphi:\partial T_X^L\to \partial T_2^M$ between the boundaries of the towers $T_X^L$ and $T_2^M$, which are bi-uniformly equivalent to $X$ and $2^{<\IZ}$, respectively.

\section{Proof of Theorem~\ref{mU-char}  (Micro-Uniform Characterization of the Cantor bi-cube)}

The ``only if'' part of Theorem~\ref{mU-char} easily follows from Lemmas~\ref{l1} and \ref{l2}. To prove the ``if'' part, it suffices to prove that any two non-compact complete metric spaces $X,Y$ of micro-uniform dimension zero are micro-uniformly equivalent if there is  $\e\in(0,1)$ is such that $\Theta^\e_\delta(X)$ and $\Theta^\e_\delta(Y)$ are finite for all positive $\delta\le \e$ and $\lim\limits_{\delta\to+0}\theta^\e_\delta(X)=\infty=\lim\limits_{\delta\to+0}\theta^\e_\delta(Y)$.

Being complete and not compact, the spaces $X$ and $Y$ are not totally bounded. Consequently, there is $\e_0\in(0,1)$ so small that $X$ cannot be covered by a finite number of sets of diameter $<\e_0$. Since $X$ has micro-uniform dimension zero, we can take a positive $\e$ so small that each $\e$-connected component $C_\e(x)$, $x\in X$, has diameter $<\e_0$. Then the choice of $\e_0$  guarantees that the cover $\C_\e(X)=\{C_\e(x):x\in X\}$ is infinite. Since $X$ is separable the cover $\C_\e(X)$ is countable.

By the same reason, we can assume that $\e$ is so small that $\C_\e(Y)=\{C_\e(y):y\in Y\}$ is a countable infinite cover of $Y$ consisting of sets of diameter $<\e_0$.

It is clear that the metric space $X$ is micro-uniformly equivalent to $X$ endowed with the metric $\min\{1,d_X\}$. So, we lose no generality assuming that $d_X\le 1$. By the same reason, we can assume that $d_Y\le 1$. In this case we shall prove that the bounded metric spaces $X,Y$ are bi-uniformly equivalent.

Let $\alpha_0=\beta_0=\e$ and $\alpha_k=\beta_k=k$ for $k\in\IN$. By reverse induction, construct sequences $(\alpha_{k})_{k=-\infty}^{-1}$ and $(\beta_k)_{k=-\infty}^{-1}$ of real numbers in the interval $(0,1)$ such that
\begin{itemize}
\item[(i)] $\alpha_{k-1}<\alpha_{k}$ and $\beta_{k-1}<\beta_k$ for each $k\le 0$;
\item[(ii)] $\lim_{k\to-\infty}\alpha_k=0$, $\lim_{k\to-\infty}\beta_k=0$ and
\item[(iii)] $\theta^{\alpha_{k}}_{\alpha_{k-1}}(X)\ge \Theta^{\beta_{k}}_{\beta_{k-1}}(Y)$;
\item[(iv)] $\theta^{\beta_{k}}_{\beta_{k-1}}(Y)\ge \Theta^{\alpha_{k+1}}_{\alpha_k}(X)$.
\end{itemize}

For the level set $A=\{\alpha_k:k\in \IZ\}$  consider the canonical $A$-tower $T_X^A=\{(C_\lambda(x),\lambda):x\in X,\;\lambda\in A\}$ of the metric space $X$.
The level set $\Lev(T_X^A)$ of the tower $T^A_X$ can be identified with the set $A$.
By Corollary~\ref{c5}(3), the canonical map
$$C_A:X\to\partial T_X^A,\;\;C_A:x\mapsto C_A(x)=\{(C_\lambda(x),\lambda):\lambda\in A\},$$
is a bi-uniform equivalence. The choice of $\alpha_0=\e$ guarantees that the zeros level $Lev_0(T_X^A)=\{(C_\lambda(x),\lambda):x\in X,\;\lambda=\alpha_0\}\subset T_X^A$ is countable.
On the other hand, $d_X\le 1$ implies that for each $k\in\IN$ the level $\Lev_k(T_X^A)=\{(C_{\alpha_k}(x),\alpha_k):x\in X\}=\{(X,k)\}$ is a singleton.

By analogy, for the level set $B=\{\beta_k:k\in \IZ\}$  consider the canonical $B$-tower $T_Y^B=\{(C_\lambda(y),\lambda):y\in Y,\;\lambda\in B\}$ of the metric space $Y$. By Corollary~\ref{c5}(3), the canonical map
$$C_B:Y\to\partial T_Y^B,\;\;C_B:y\mapsto C_B(y)=\{(C_\lambda(y),\lambda):\lambda\in B\},$$
is a bi-uniform equivalence. The choice of $\beta_0=\e$ guarantees that the zeros level $\Lev_0(T_Y^B)=\{(C_\lambda(y),\lambda):y\in Y,\;\lambda=\beta_0\}\subset T_Y^B$ is countable.
On the other hand, $d_Y\le 1$ implies that for each $k\in\IN$ the level $\Lev_k(T_Y^B)=\{(C_{\beta_k}(y),\beta_k):y\in Y\}=\{(Y,k)\}$ is a singleton.

For every $k\in\IZ$ consider the sets $A_k=\{\alpha_n:n\ge k\}$ and
$B_k=\{\alpha_n:n\ge k\}$. Let $\varphi_1:T_X^{A_1}\to T_Y^{B_1}$ be the tower isomorphism assigning to the unique point $(X,k)$ of a level of $T_X^{A_0}$ the unique point $(Y,k)$ of the corresponding level of the tower $T_Y^{B_1}$. Since the 0th levels
of the towers $T_X^{A_0}$ and $T_Y^{B_0}$ both are countably infinite, we can extend the tower isomorphism $\varphi_1$ to a tower isomorphism $\varphi_0:T_X^{A_0}\to T_Y^{B_0}$.

By analogy with the proof of Theorem~\ref{MU-char}, by the reverse induction we can construct a sequence of surjective tower immersions $\varphi_k:T_X^{A_k}\to T_Y^{B_k}$, $k\le 0$ such that $\varphi_{k-1}|T_X^{A_k}=\varphi_k$ for all $k\le 0$. Those tower immersions compose a surjective tower immersion $\varphi:T_X^A\to T_Y^B$ such that $\varphi|T_X^{A_k}=\varphi_k$ for all $k\le 0$. By Proposition~\ref{p8}, the immersion $\varphi$ induces a micro-uniform equivalence $\partial \varphi:\partial T_X^A\to \partial T_Y^B$. By Corollary~\ref{c5}(3), the boundary $\partial T_X^A$ is bi-uniformly equivalent to $X$ while $\partial T_X^B$ is bi-uniformly equivalent to $Y$. Consequently, the (bounded) metric spaces $X$ and $Y$ are bi-uniformly equivalent.

\section{Proof of Theorem~\ref{univer} (The Universality of the Cantor bi-cube)}

The ``only if'' part easily follows from Lemmas~\ref{l1} and \ref{l2}.

To prove the ``if'' part, assume that $X$ has bi-uniform dimension zero and $\Theta^\e_\delta(X)$ is finite for all $0<\delta<\e<\infty$. Since the completion of $X$ has the same properties, we lose no generality assuming that the space $X$ is complete.

For the level set $L=\{2^n:n\in\IZ\}$ consider the canonical $L$-tower $T_X^L$ of $X$. By
Corollary~\ref{c5}(3), the canonical map $C_L:X\to \partial T_X^L$ is a bi-uniform equivalence. It follows that $\Deg_{2^n}^{2^{n+1}}(T^L_X)=\Theta_{2^n}^{2^{n+1}}(X)<\infty$ for all $2^n\in L=\Lev(T^L_X)$.

Let $T_\w$ be a homogeneous tower such that the set $\Lev(T_\w)$ is order isomorphic to $\IZ$ and $\deg(x)=\w$ for each $x\in T$. Let $f:\Lev(T^L_X)\to\Lev(T_\w)$ be an order isomorphism. By induction construct a homogeneous subtower $T\subset T_\w$ such that $$\Deg_{\lambda}^{\lambda+1}(T)=\max\{2,\Deg_{f^{-1}(\lambda)}^{f^{-1}(\lambda+1)}(T_X)\}.$$

By Proposition~\ref{p11}, there is a tower embedding $\varphi:T^L_X\to T$ such that $\varphi_\Lev=f$. By Proposition~\ref{p8}(2) the tower embedding $\varphi$ induces a bi-uniform embedding $\partial \varphi:\partial T^L_X\to\partial T$. By Theorem~\ref{bU-char}, the boundary $\partial T$ of the homogeneous $\dupa$-unbounded tower $T$ is bi-uniformly equivalent to the Cantor bi-cube $2^{<\IZ}$. Since $X$ is bi-uniformly equivalent to $\partial T^L_X$, we see that $X$ bi-uniformly embeds into $2^{<\IZ}$.

\section{Proof of Theorem~\ref{t5}}

Let $X$ be an isometrically homogeneous countable proper metric space of asymptotic dimension zero.
The Baire Theorem guarantees that $X$ has an isolated point and then the isometric homogeneity of $X$ implies that $X$ is uniformly discrete in the sense that for some $\e>0$ all $\e$-balls in $X$ are singletons. Being proper and uniformly discrete, the space $X$ is boundedly-finite. Since $X$ has asymptotic dimension zero, each $\e$-connected component $C_\e(x)\subset X$ is bounded and hence finite.

So, we can consider the function $f_X:\Pi\to\w\cup\{\infty\}$ assigning to each prime number $p\in\Pi$ the (finite or infinite) number
$$f_X(p)=\sup\{k\in\w:\mbox{$p^k$ divides $|C_\e(x)|$ for some $\e>0$ and $x\in X$}\}.$$

Given a function $f:\Pi\to\w\cup\{\infty\}$ consider the direct sum
$$\IZ_f=\oplus_{p\in\Pi}\IZ_p^{f(p)}$$ of cyclic groups $\IZ_p=\IZ/p\IZ$.

In \cite{Smith} J.Smith proved that each countable group admits a proper left-invariant metric and that for any two proper left-invariant metrics $\rho$, $d$ on $G$ the identity map $\id:(G,\rho)\to(G,d)$ is a bi-uniform equivalence. In the sequel we shall endow each countable group $G$ (in particular, each group $\IZ_f$) with a proper left-invariant metric.

\begin{lemma}\label{l7} Each isometrically homogeneous proper countable metric space $X$ of asymptotic dimension zero is bi-uniformly equivalent to the group $\IZ_{f_X}$.
\end{lemma}

\begin{proof} Consider the canonical $\w$-tower $T^{\w}_X=\{(C_n(x),n):x\in X,\;n\in\w\}$ of the metric space $X$.

Taking into account that each $0$-connected component $C_0(x)$ coincides with the singleton $\{x\}$ and applying Corollary~\ref{c5}, we conclude that canonical map $C_\w:X\to\partial T_X^{\w}$ is a bi-uniform equivalence. The isometric homogeneity of the metric space $X$ implies the homogeneity of the tower $T_X^\w$. It follows that for every $n\in\w$ we the degree
$$\deg_n(T^\w_X)=\deg_n(T^\w_X)=|C_{n+1}(x)/\C_n(X)|$$equals the number of $n$-connected components of $X$ composing an $(n+1)$-connected component of $X$.

For every $n\in\w$ let $f_n:\Pi\to\w$ be the function assigning to each prime number $p$ the maximal number $k\ge 0$ such that $p^k$ divides $\deg_n(T_X)$. Then the group $\IZ_{f_n}$ is finite and has order $|\IZ_{f_n}|=\deg_n(T_X)$.

Consider the group $G=\oplus_{n\in\w}\IZ_{f_n}$ and observe that it is isomorphic (with help of a coordinate permutating isomorphism) to the group $\IZ_{f_X}$. The group $G$ can be written as the union $G=\bigcup_{m\in\w}G_m$ of an increasing sequence $(G_m)_{m\in\w}$ of subgroups where $G_0=\{0\}$ and $G_m=\bigcup_{n=0}^{m-1}\IZ_{f_n}$ for $m>0$.

Consider the ${\da}$-bounded tower $T_G=\{xG_m:x\in G,\;m\in\w\}$ endowed with the inclusion relation and observe that it is homogeneous and $\deg_n(T_G)=|\IZ_{f_n}|=\deg_n(T_X)$ for all $n\in\w$. By Proposition~\ref{p11}, there is a tower isomorphism $\varphi:T^\w_X\to T_G$ inducing a bi-uniform equivalence $\partial \varphi:\partial T^\w_X\to\partial T_G$. Then the bi-uniform equivalence between $X$ and $\IZ_f$ is obtained as the composition of the bi-uniform equivalences:
$$X\sim \partial T^\w_X\sim\partial T_G\sim G\sim \IZ_f.$$
\end{proof}

The following lemma (that essentially is due to I.Protasov \cite{Pro}), combined with Lemma~\ref{l7} imply Theorem~\ref{t5}.

\begin{lemma} If two countable proper isometrically homogeneous metric spaces $X,Y$ of asymptotic dimension zero are bi-uniformly equivalent, then $f_X=f_Y$.
\end{lemma}

\begin{proof} Since $X$ and $Y$ are boundedly-finite spaces of asymptotic dimension zero their $\e$-connected components are finite for all $\e<\infty$.

The inequality $f_X\le f_Y$ will follow as soon as
we check that for each prime number $p$ and each $k\in\IN$ if $p^k$ divides the cardinality $|C_\e(x)|$ for some $x\in X$ and $\e<\infty$, then $p^k$ divides $|C_\delta(y)|$ for some $\delta<\infty$ and $y\in Y$.

 Let $\varphi:X\to Y$ is a bi-uniform equivalence and $\delta=\w_\varphi(\e)$.
By Lemma~\ref{l1}, the image $\varphi(C_\e(x))$ of the $\e$-connected component $C_\e(x)$ lies in the $\delta$-connected component $C_\delta(y)$ of the point $y=\varphi(x)$ in $Y$. Consider the preimage $A=\varphi^{-1}(C_\delta(y))$ and observe that by Lemma~\ref{l1} for each point $a\in A$ we get $\varphi(C_\e(a))\subset C_\delta(\varphi(a))=C_\delta(y)$ (the latter equality holds because $\varphi(a)\in C_\delta(y)$). Consequently, $C_\e(y)\subset A$. This implies that $A$ decomposes into a disjoint union of $\e$-connected components of $X$. Since the metric space $X$ is isometrically homogeneous, any two $\e$-connected component of $X$ have the same cardinality. Consequently, $|C_\e(x)|$ divides $|A|=|C_\delta(y)|$. Since $p^k$ divides $|C_\e(x)|$ it also divides $|C_\delta(y)|$. This concludes the proof of the inequality $f_X\le f_Y$.

The inequality $f_Y\le f_X$ can be proved by analogy.
\end{proof}

\section{Acknowledgment}

The authors would like to express their sincere thanks to Igor Volodymyrovych Protasov and Jose Manuel Higes L\'opez for fruitful discussions and useful suggestions.

\end{document}